\documentclass[]{article}

\usepackage{amsmath,amsthm,amssymb,bbm,float}
\usepackage[margin=1.15in]{geometry}
\usepackage[title]{appendix}
\usepackage{pdflscape}

\theoremstyle{theorem}
\newtheorem{theorem}{Theorem}[subsection]
\newtheorem{proposition}[theorem]{Proposition}
\newtheorem{lemma}[theorem]{Lemma}
\newtheorem{corollary}[theorem]{Corollary}

\newtheorem{question}[theorem]{Question}
\theoremstyle{definition}
\newtheorem{note}[theorem]{Note}
\newtheorem{example}[theorem]{Example}

\title{Modular tensor categories, subcategories, and Galois orbits}
\date{\today}
\author{Julia Plavnik, Andrew Schopieray, \\ Zhiqiang Yu, and Qing Zhang}

\begin{document}

\maketitle

\begin{abstract}
We establish a set of general results to study how the Galois action on modular tensor categories interacts with fusion subcategories.  This includes a characterization of fusion subcategories of modular tensor categories which are closed under the Galois action, and a classification of modular tensor categories which factor as a product of pointed and transitive categories in terms of pseudoinvertible objects.  As an application, we classify modular tensor categories with two Galois orbits of simple objects and a nontrivial grading group.  
\end{abstract}


\section{Introduction}

\par Modular tensor categories are an algebraic shadow of conformal field theory, and instantiate an intersection of representation theory, low-dimensional topology, number theory, and mathematical physics.  From their definition it seems strange that a nontrivial example of such a heavily-defined object would even exist, yet there are incredibly diverse infinite families coming from the representation theory of finite groups, quantum groups at roots of unity, and categorical generalizations thereof.  There are two open-ended and interrelated approaches to understanding modular tensor categories: to produce novel examples, and to organize known examples.  Often, novel examples inspire the classification of examples with particular characteristics while classification results demonstrate unusual gaps in the library of known examples.  Here we contribute results organizing modular tensor categories by the interaction between fusion subcategories and the Galois action on simple objects.

\par The Galois action on a modular tensor category associates a permutation of the finite set of simple objects to every Galois automorphism in the absolute Galois group $\mathrm{Gal}(\overline{\mathbb{Q}}/\mathbb{Q})$.  As the modular data of a modular tensor category consists of cyclotomic integers, it is sufficient to study the Galois action of $\mathrm{Gal}(\mathbb{Q}(\zeta)/\mathbb{Q})$ where $\zeta$ is a primitive root of unity whose order is the conductor, or Frobenius-Schur exponent of the modular tensor category in question.  Previously \cite[Theorem II]{2020arXiv200701366N}, it was shown there is a single infinite family of modular tensor categories with a transitive Galois action, which is to say that there is a unique orbit of simple objects under this action.  One reason transitive modular tensor categories have a particularly elegant description is that they all possess an essentially unique factorization into a product of simple transitive modular tensor categories, in the sense that each factor has no proper nontrivial fusion subcategories.  But in general it is necessary to have a more nuanced understanding of the interaction between fusion subcategories and the Galois action.

\par In \cite{2019arXiv191212260G}, number-theoretical properties of Frobenius-Perron dimension were used to identify and classify fusion subcategories of fusion categories.  Similarly, here we consider centralizing simple objects in a modular tensor category as a numerical constraint, and apply analogous methods.  In Theorem \ref{prelem0}, we prove that a fusion subcategory of a modular tensor category is closed under the Galois action if and only if its centralizer subcategory is integral, which means all simple objects have integer Frobenius-Perron dimension.  This implies, for example, that all fusion subcategories of integral modular tensor categories (e.g., the twisted doubles of finite groups) are preserved under the Galois action.   A common source of simple objects of integer Frobenius-Perron dimension are the invertible objects, i.e.\ those with Frobenius-Perron dimension 1.  In Theorem \ref{what}, we obtain a lower bound for $|\mathrm{Orb}(\mathcal{C})|$, the number of Galois orbits of a modular tensor category $\mathcal{C}$, based on the prime factorization of the dimension of its maximal pointed fusion subcategory $\mathcal{C}_\mathrm{pt}$.  A weaker condition than invertibility is requiring a simple object to have categorical dimension $\pm1$, which we call pseudoinvertibility.  We prove in Theorem \ref{what2} that a modular tensor category $\mathcal{C}$ factorizes as $\mathcal{P}\boxtimes\mathcal{T}$ where $\mathcal{P}$ is a pointed modular tensor category and $\mathcal{T}$ is a transitive modular tensor category if and only if every Galois orbit of simple objects contains a pseudoinvertible object.  In terms of formal codegrees \cite{MR2576705}, these conditions are equivalent to all formal codegrees of a modular tensor category being Galois conjugate.

\par Lastly, we utilize our general results to further the classification of modular tensor categories with a small number of Galois orbits of simple objects initiated in \cite{2020arXiv200701366N}.  When $|\mathrm{Orb}(\mathcal{C})|=1$, i.e.\ $\mathcal{C}$ is transitive, then $\mathcal{C}_\mathrm{pt}$ is trivial, or equivalently $\mathcal{C}$ has a trivial universal grading group \cite[Section 4.14]{tcat}.   Transitive modular tensor categories are also self-dual and Galois conjugate to pseudounitary categories.  Modular tensor categories $\mathcal{C}$ with $|\mathrm{Orb}(\mathcal{C})|=2$ are more complex in all of these regards, and so their classification is naturally partitioned into first studying those which have nontrivial gradings, then those with nontrivial fusion subcategories.  Our general results imply that

\begin{itemize}
\item[-] $|\mathrm{Orb}(\mathcal{C})|=2$ and $\mathcal{C}_\mathrm{pt}$ is nontrivial if and only if $\mathcal{C}\simeq\mathcal{D}\boxtimes\mathcal{T}$ where $\mathcal{T}$ is a transitive modular tensor category, and $\mathcal{D}$ is pointed of prime dimension coprime to the conductor of $\mathcal{T}$, or an Ising category (Corollary \ref{pointedprop});
\item[-] $|\mathrm{Orb}(\mathcal{C})|=2$ and $\mathcal{C}_\mathrm{pt}$ is trivial if and only if $\mathcal{C}\simeq\mathcal{D}\boxtimes\mathcal{T}$ where $\mathcal{T}$ is a transitive modular tensor category, and $\mathcal{D}$ is a simple modular tensor category whose conductor is coprime to that of $\mathcal{T}$, or $\mathrm{rank}(\mathcal{T})>2$ and $\mathcal{D}$ is braided equivalent to $\mathcal{F}_1\boxtimes\mathcal{F}_2$ where $\mathcal{F}_1,\mathcal{F}_2$ are any of the Fibonacci modular tensor categories (Proposition \ref{unpointedprop}).
\end{itemize}

Therefore, a complete classification is reduced to understanding those examples which are simple, which will require future analysis using other methods.  We provide two infinite families of simple modular tensor categories $\mathcal{C}$ with $|\mathrm{Orb}(\mathcal{C})|=2$ (Figure \ref{fig:two}) which are not pointed; one family is self-dual while the other is not self-dual.  It is unclear whether 5 additional sporadic examples of modular tensor categories of this type belong to infinite families or if they are truly exceptional.  Note that 3 of these examples are not pseudounitary, and for at least one, there does not exist a pseudounitary modular tensor category nor a pseudounitary fusion category with these fusion rules \cite{schopieray2020nonpseudounitary}.

\par A brief introduction to modular tensor categories, their Galois action, and corresponding modular group representation are given in Section \ref{sec:prelim}; we provide an expansive collection of examples in Section \ref{sec:examples}.  Figure \ref{fig:notation} below contains notation which recurs throughout.  Our general results pertaining to the Galois action are contained in Section \ref{sec:gen}, while the particular application to modular tensor categories with 2 Galois orbits is the subject of Section \ref{sec:two}.  Appendix \ref{tables}, consisting of the $\mathfrak{t}$-spectra of irreducible representations of $\mathrm{SL}(2,\mathbb{Z}/N\mathbb{Z})$ for $N\in\mathbb{Z}_{\geq2}$, is necessary for the proofs of Section \ref{sec:two}.

\begin{figure}[H]
\centering
\begin{equation*}
\begin{array}{|c|c|c|}
\hline \mathrm{Notation} & \mathrm{Parameters} & \mathrm{Meaning} \\\hline\hline
\phi(n) & n\in\mathbb{Z}_{\geq1} & \text{number of positive integers }m\text{ with }\mathrm{gcd}(m,n)=1 \\\hline
\zeta_n & n\in\mathbb{Z}_{\geq1} & \exp(2\pi i/n) \\\hline
\mathbb{Q}(\zeta_n)^+ & n\in\mathbb{Z}_{\geq1} & \text{maximal totally real subfield of }\mathbb{Q}(\zeta_n) \\\hline
[\mathbb{K}:\mathbb{L}] & \text{algebraic number fields }\mathbb{L}\subset\mathbb{K} & \text{degree of }\mathbb{K}\text{ as a vector space over }\mathbb{L}\\\hline
\mathcal{O}(\mathcal{C}) & \text{fusion category }\mathcal{C} & \text{set of isomorphism classes of simple objects of }\mathcal{C} \\\hline 
C_\mathcal{C}(\mathcal{D}) & \text{braided fusion categories }\mathcal{D}\subset\mathcal{C} & \text{relative centralizer of }\mathcal{D}\subset\mathcal{C} \\\hline
\mathcal{O}_X & \text{simple object }X\text{ in} & \text{Galois orbit of }X\text{ in }\mathcal{C} \\
 &\text{a modular tensor category }\mathcal{C} & \\\hline
\mathrm{Orb}(\mathcal{C}) &\text{ modular tensor category }\mathcal{C} & \text{the set of Galois orbits of simple objects of }\mathcal{C} \\\hline
\end{array}
\end{equation*}
    \caption{Recurring notation}%
    \label{fig:notation}%
\end{figure}


\section{Preliminaries}\label{sec:prelim}

Our main objects of study are modular tensor categories.  We only briefly recount the components of a modular tensor category with an aim to reach our main tool: the modular data of a modular tensor category and its Galois action.  For comprehensive detail one can refer to the standard text \cite[Sections 8.13-8.17]{tcat} and references within. The structure of a modular tensor category is built upon a fusion category: a $\mathbb{C}$-linear, semisimple, rigid monoidal category (the monoidal operation will be denoted $\otimes$) with finitely-many isomorphism classes of simple objects and a simple $\otimes$-unit which will be denoted $\mathbbm{1}$.  We will denote the set of isomorphism classes of simple objects of a fusion category $\mathcal{C}$ by $\mathcal{O}(\mathcal{C})$.  Rigidity of a fusion category $\mathcal{C}$ is the existence of a suitable notion of duality which we will denote $X^\ast$ for an object $X\in\mathcal{C}$.  Upon the base fusion category, a modular tensor category includes two additional structures: a nondegenerate braiding (see \cite[Section 8.1]{tcat} and Section \ref{prestruct} below), or a family of natural isomorphisms dictating the commutativity of $\otimes$, and a spherical structure \cite[Section 4.7]{tcat}, or a family of natural isomorphisms $X\to X^{\ast\ast}$ for all $X\in\mathcal{C}$ which allow a well-defined notion of trace for endomorphisms in the category.


\subsection{The modular data}

Let $\mathcal{C}$ be a modular tensor category.  There are two important symmetric invertible $|\mathcal{O}(\mathcal{C})|\times|\mathcal{O}(\mathcal{C})|$ matrices associated to $\mathcal{C}$, denoted $s$ and $t$ and often referred to as the modular data of $\mathcal{C}$.  The matrix $s$ consists of the traces (inherited from the chosen spherical structure) of the double braidings on pairs of simple objects.  For example, the double braiding of $\mathbbm{1}$ and any other $X\in\mathcal{O}(\mathcal{C})$ is the identity, and so its trace $s_{\mathbbm{1},X}=s_{X,\mathbbm{1}}$ is aptly referred to as $\dim(X)$, the categorical dimension of $X$.  The matrix $t$ is a diagonal matrix whose diagonal entries are roots of unity $t_X$ known as the twists or full twists of simple objects $X$; the order of the $t$-matrix is known as the conductor of $\mathcal{C}$.  If $\mathcal{C}$ and $\mathcal{D}$ are modular tensor categories whose conductors are coprime as integers, we will simply refer to $\mathcal{C}$ and $\mathcal{D}$ as coprime for brevity.

\par The namesake of the modular data is its relation to the modular group $\mathrm{SL}(2,\mathbb{Z})$.  Recall that the modular group $\mathrm{SL}(2,\mathbb{Z})$, the group of $2\times2$ integer matrices with determinant 1, is generated by
\begin{equation}
\mathfrak{s}:=\left[\begin{array}{cc}0 & -1 \\ 1 & 0\end{array}\right]\qquad\text{and}\qquad\mathfrak{t}:=\left[\begin{array}{cc}1 & 1 \\ 0 & 1\end{array}\right].
\end{equation}
If $\mathcal{C}$ is a modular tensor category, then $\mathfrak{s}\mapsto s$, $\mathfrak{t}\mapsto t$ defines a projective representation of $\mathrm{SL}(2,\mathbb{Z})$ \cite[Section 8.16]{tcat}.  Depending on the application, one often requires a normalized version of the modular data, giving a linear representation of $\mathrm{SL}(2,\mathbb{Z})$.  This normalization can be made very explicit.  Define
\begin{equation}
\dim(\mathcal{C}):=\sum_{X\in\mathcal{O}(\mathcal{C})}s_{\mathbbm{1},X}^2\qquad\text{and}\qquad\tau:=\sum_{X\in\mathcal{O}(\mathcal{C})}t_Xs_{\mathbbm{1},X}^2.
\end{equation}
Let $\dim(\mathcal{C})^{1/2}$ be the positive square root, and let $\gamma$ be any third root of $\xi(\mathcal{C}):=\tau\cdot\dim(\mathcal{C})^{-1/2}$.  The root of unity $\xi(\mathcal{C})$ is usually called the multiplicative central charge of $\mathcal{C}$.  Then the normalized modular data $\tilde{s}:=\dim(\mathcal{C})^{-1/2}s$ and $\tilde{t}:=\gamma^{-1}t$ give a linear representation of $\mathrm{SL}(2,\mathbb{Z})$ with $\mathfrak{s}\mapsto\tilde{s}$ and $\mathfrak{t}\mapsto\tilde{t}$.  We will elaborate on this representation in Section \ref{prerep}.


\subsection{The Galois action}\label{secgal} 

If $\mathcal{C}$ is a modular tensor category, the characters of the Grothendieck ring $K(\mathcal{C})$ are in bijection with $\mathcal{O}(\mathcal{C})$ \cite[Section 8.14]{tcat}.  Specifically, for each $X\in\mathcal{O}(\mathcal{C})$, $Y\mapsto s_{Y,X}/s_{\mathbbm{1},X}$ is a ring homomorphism $K(\mathcal{C})\to\mathbb{C}$ and every character of $K(\mathcal{C})\to\mathbb{C}$ arises in this way.  This bijection gives a well-defined Galois action on $\mathcal{O}(\mathcal{C})$ coming from the fact that the absolute Galois group $\mathrm{Gal}(\overline{\mathbb{Q}}/\mathbb{Q})$ permutes the characters of $K(\mathcal{C})$.  In particular, for any Galois automorphism $\sigma\in\mathrm{Gal}(\overline{\mathbb{Q}}/\mathbb{Q})$, there exists a unique permutation $\hat{\sigma}:\mathcal{O}(\mathcal{C})\to\mathcal{O}(\mathcal{C})$ defined by
\begin{equation}\label{gal}
\sigma\left(\dfrac{s_{X,Y}}{s_{\mathbbm{1},Y}}\right)=\dfrac{s_{X,\hat{\sigma}(Y)}}{s_{\mathbbm{1},\hat{\sigma}(Y)}}
\end{equation}
for all $X,Y\in\mathcal{O}(\mathcal{C})$.  One then deduces, for example, that for all $X\in\mathcal{O}(\mathcal{C})$ and $\sigma\in\mathrm{Gal}(\overline{\mathbb{Q}}/\mathbb{Q})$,
\begin{equation}\label{dims}
\dim(\hat{\sigma}(X))^2=s_{\mathbbm{1},\hat{\sigma}(X)}^2=\dfrac{\dim(\mathcal{C})}{\sigma(\dim(\mathcal{C}))}\sigma(s_{\mathbbm{1},X}^2)=\dfrac{\dim(\mathcal{C})}{\sigma(\dim(\mathcal{C}))}\sigma(\dim(X))^2.
\end{equation}
The Galois action on the matrix $t$ is by squares.  We have $\sigma^2(t_X)=\sigma^2(\gamma)\gamma^{-1}\cdot t_{\hat{\sigma}(X)}$ \cite[Theorem II(iii)]{dong2015congruence} which on the normalized modular data simplifies to $\sigma^2(\tilde{t}_X)=\tilde{t}_{\hat{\sigma}(X)}$.  For this reason we will often refer to roots of unity $\zeta_1,\zeta_2$ as being square Galois conjugate, i.e.\ there exists $\sigma\in\mathrm{Gal}(\overline{\mathbb{Q}}/\mathbb{Q})$ such that $\sigma^2(\zeta_1)=\zeta_2$.  Denote the Galois orbit of $X\in\mathcal{O}(\mathcal{C})$ by $\mathcal{O}_X$ and the set of Galois orbits of $\mathcal{O}(\mathcal{C})$ by $\mathrm{Orb}(\mathcal{C})$.  If $\mathcal{O}(\mathcal{C})=\mathcal{O}_\mathbbm{1}$, i.e. $|\mathrm{Orb}(\mathcal{C})|=1$, we say $\mathcal{C}$ is transitive.  Transitive modular tensor categories are completely classified \cite[Theorem II]{2020arXiv200701366N}.
\par It is well-known that the entries of the $s$-matrix are cyclotomic integers \cite[Theorem 8.14.7]{tcat} so the Galois action on a modular tensor category is determined by the permutations arising from the Galois group of a cyclotomic field.  One can always let this field be $\mathbb{Q}(\zeta_N)$ where $\zeta_N=\exp(2\pi i/N)$ and $N$ is the order or conductor of the $t$-matrix, as $s_{X,Y}\in\mathbb{Q}(\zeta_N)$ for all $X,Y\in\mathcal{O}(\mathcal{C})$ \cite[Section 5]{MR2725181}.  In particular the character values of the Grothendieck ring of $\mathcal{C}$ lie in $\mathbb{Q}(\zeta_N)$; the degree of the extensions generated by these so-called Verlinde eigenvalues determine the size of the Galois orbits.   To this end, define the field of Verlinde eigenvalues for each $Y\in\mathcal{O}(\mathcal{C})$ by $\mathbb{L}_Y:=\mathbb{Q}(s_{X,Y}/s_{\mathbbm{1},Y}:X\in\mathcal{O}(\mathcal{C}))$ so that the following result follows almost by the definition of the Galois action. 
\begin{lemma}{\textnormal{\cite[Lemma 3.2]{MR3632091}}}\label{lem1}
Let $\mathcal{C}$ be a modular tensor category.  For all $X\in\mathcal{O}(\mathcal{C})$, $|\mathcal{O}_X|=[\mathbb{L}_X:\mathbb{Q}]$.
\end{lemma}
One should not confuse the Galois action on modular tensor categories with Galois conjugacy of fusion, braided fusion, and modular tensor categories.  If $\mathcal{C}$ is a fusion, braided fusion, or modular tensor category, and $\sigma\in\mathrm{Gal}(\overline{\mathbb{Q}}/\mathbb{Q})$ is any Galois automorphism, then we denote by $\mathcal{C}^\sigma$ the category constructed by applying $\sigma$ to all structural constraints of $\mathcal{C}$.    The fusion categories $\mathcal{C}$ and $\mathcal{C}^\sigma$ have identical fusion rules.


\subsection{Structure of modular tensor categories}\label{prestruct}

Perhaps the most important character of the Grothendieck ring of a fusion category is the Frobenius-Perron dimension $X\mapsto\mathrm{FPdim}(X)$ where $\mathrm{FPdim}(X)$ is defined as the maximal real eigenvalue of the fusion matrix $(N_{X,Y}^Z)_{Y,Z\in\mathcal{O}(\mathcal{C})}$ with $N_{X,Y}^Z:=\dim_\mathbb{C}\mathrm{Hom}(X\otimes Y,Z)$.  We denote by $\mathcal{C}_\mathrm{pt}$, the fusion subcategory of $\mathcal{C}$ generated by $X\in\mathcal{O}(\mathcal{C})$ with $\mathrm{FPdim}(X)=1$ and we say such $X$ are invertible.   If $\mathcal{C}$ is a modular tensor category, or more generally a spherical fusion category, and $\dim(X)=\pm1$ for some $X\in\mathcal{O}(\mathcal{C})$, we refer to $X$ as pseudoinvertible because these notions are equivalent for pseudounitary fusion categories \cite[Corollary 9.6.6]{tcat}, i.e.\ fusion categories such that $\mathrm{FPdim}(\mathcal{C})=\dim(\mathcal{C})$.  The universal grading group of a modular tensor category is isomorphic to $\mathcal{O}(\mathcal{C}_\mathrm{pt})$ \cite[Theorem 6.3]{nilgelaki} and the trivial component of this grading is $\mathcal{C}_\mathrm{ad}$, the fusion subcategory of $\mathcal{C}$ $\otimes$-generated by $X\otimes X^\ast$ over all $X\in\mathcal{O}(\mathcal{C})$.

\par If $\mathcal{C}$ is a braided fusion category and $\mathcal{D}\subset\mathcal{C}$ is a fusion subcategory, we denote the relative centralizer of $\mathcal{D}$ in $\mathcal{C}$ by $C_\mathcal{C}(\mathcal{D})$.  The relative centralizer of $\mathcal{D}\subset\mathcal{C}$ is the fusion subcategory consisting of $X\in\mathcal{O}(\mathcal{C})$ which centralize all simple objects of $\mathcal{D}$, i.e.\ the double braidings with all other $Y\in\mathcal{O}(\mathcal{D})$ are trivial.  If $\mathcal{C}$ is a modular tensor category, $X,Y\in\mathcal{O}(\mathcal{C})$ centralize one another if and only if $s_{X,Y}=\dim(X)\dim(Y)$ \cite[Proposition 8.20.5(i)]{tcat}, a more important characterization in this exposition.  We reserve the notation $\mathcal{C}':=C_{\mathcal{C}}(\mathcal{C})$ for the symmetric center of $\mathcal{C}$, but the reader may find this notation used for generic centralizers in other sources.  Nondegenerate braidings, a necessary condition for modular tensor categories, are those braided fusion categories such that $\mathcal{O}(\mathcal{C}')=\{\mathbbm{1}\}$ while symmetrically braided fusion categories are those for which $\mathcal{C}'=\mathcal{C}$.  Another useful characterization is that $\mathcal{C}_\mathrm{ad}=C_\mathcal{C}(\mathcal{C}_\mathrm{pt})$ for all modular tensor categories $\mathcal{C}$ \cite[Corollary 6.9]{nilgelaki}.  It is well-known that if $\mathcal{C}$ is a nondegenerately braided fusion category and $\mathcal{D}\subset\mathcal{C}$ is a nondegenerately braided fusion subcategory, then $\mathcal{C}\simeq\mathcal{D}\boxtimes C_\mathcal{C}(\mathcal{D})$ \cite[Proposition 2.2]{DMNO} where $\boxtimes$ is the Deligne tensor product.  As a consequence, each nondegenerately braided fusion category has a factorization into prime factors (i.e.\ containing no nondegenerately braided fusion subcategories) which is unique up to a permutation of factors when $\mathcal{C}_\mathrm{pt}$ is trivial \cite[Proposition 2.2]{DMNO}.


\subsection{Representation theory of $\mathrm{SL}(2,\mathbb{Z}/N\mathbb{Z})$}\label{prerep}

Let $\mathcal{C}$ be a modular tensor category, $\tilde{\rho}_\mathcal{C}$ be its associated $\mathrm{SL}(2,\mathbb{Z})$ representation defined by the normalized modular data $\tilde{s},\tilde{t}$.  It was proven in \cite[Theorem 6.8]{MR2725181} that $\tilde{\rho}_\mathcal{C}$ can be factored through $\mathrm{SL}(2,\mathbb{Z}/N\mathbb{Z})$ where again $N\in\mathbb{Z}_{\geq1}$ is the order of $\tilde{t}$.  In other words, to each modular tensor category $\mathcal{C}$ one associates a finite-dimensional representation $\rho_\mathcal{C}$ of the finite group $\mathrm{SL}(2,\mathbb{Z}/N\mathbb{Z})$.  The benefit being that the irreducible finite-dimensional complex representations of $\mathrm{SL}(2,\mathbb{Z}/N\mathbb{Z})$ have been described explicitly, up to isomorphism \cite{MR444787,MR444788}.  The irreducible representations of $\mathrm{SL}(2,\mathbb{Z}/N\mathbb{Z})$ can be described from the irreducible representations of $\mathrm{SL}(2,\mathbb{Z}/p^\lambda\mathbb{Z})$ for $p^{\lambda}$ dividing $N$ where $p\in\mathbb{Z}_{\geq2}$ is prime and $\lambda\in\mathbb{Z}_{\geq1}$ (refer to \cite[Section 3]{MR1354262}), which are listed in tables of Appendix \ref{tables}.   As the $\tilde{t}$-matrix of a modular tensor category $\mathcal{C}$ is a diagonal matrix with finite order, it is useful to organize the irreducible representations of $\mathrm{SL}(2,\mathbb{Z}/p^\lambda\mathbb{Z})$ by the eigenvalues of $\rho_\mathcal{C}(\mathfrak{t})=\tilde{t}$ which we will refer to as the $\mathfrak{t}$-spectra of $\rho_\mathcal{C}$. 


\section{Examples}\label{sec:examples}

Here we describe a robust collection of examples of modular tensor categories and their sets of Galois orbits to elucidate our general results in Section \ref{sec:gen} and provide modular data for sporadic examples in Section \ref{sec:two}.  One major source of examples of modular tensor categories comes from the semisimple representation theory of quantum groups at roots of unity; these are indexed by a complex finite-dimensional simple Lie algebra $\mathfrak{g}$ and positive integer $k$, and denoted $\mathcal{C}(\mathfrak{g},k)$.  The details of this construction are but ancillary to our general theory, so we refer the reader to \cite{MR4079742} for further reading.

\begin{example}[Pointed categories]\label{point}
Let $A$ be a finite abelian group of order $n\in\mathbb{Z}_{\geq1}$ and consider the pointed modular tensor category $\mathcal{C}(A,q)$ where $q:A\to\mathbb{C}^\times$ is any nondegenerate quadratic form \cite[Section 8.4]{tcat}.  We will identify $A$ with $\mathcal{O}(\mathcal{C}(A,q))$ and write the group operation (tensor product) additively.  If $g\in A$ is arbitrary and $h\in A$ has order $m\mid n$, then by \cite[Lemma 2.4]{mug1}, $s_{g,h}^m=s_{g,mh}=s_{g,\mathbbm{1}}=1$.  Hence $s_{g,h}$ is an $m^{\tiny\text{th}}$ root of unity.  The Galois orbit of $h$ is thus determined by $\mathrm{Gal}(\mathbb{Q}(\zeta_m)/\mathbb{Q})\cong(\mathbb{Z}/m\mathbb{Z})^\times$, the multiplicative group of integers modulo $m$.  Let $\sigma\in\mathrm{Gal}(\mathbb{Q}(\zeta_m)/\mathbb{Q})$ be such that $\sigma(\zeta_m)=\zeta_m^k$ for some $k\in\mathbb{Z}$ coprime to $m$.  Then for all $g\in A$,
\begin{equation}
s_{g,\hat{\sigma}(h)}=s_{g,\hat{\sigma}(h)}/s_{\mathbbm{1},\hat{\sigma}(h)}=\sigma(s_{g,h}/s_{\mathbbm{1},h})=\sigma(s_{g,h})=s_{g,h}^k=s_{g,kh}.
\end{equation}
Therefore $\hat{\sigma}(h)=kh$ as the $s$-matrix is invertible, and $\mathcal{O}_h=\{kh:\gcd(m,k)=1\}$.  This is the set of generators of the cyclic subgroup generated by $h$.  In particular $|\mathcal{O}_h|=\phi(\mathrm{ord}(h))$.  Moreover, there is a one-to-one correspondence between $\mathrm{Orb}(\mathcal{C}(A,q))$ and the number of distinct cyclic subgroups of $A$.  While methods for computing this number have been known for almost a century \cite{MR1561347}, we prefer a more modern formula due to T\'oth \cite{MR2884710,MR2963406}.  In particular, with invariant factor decomposition $A\cong\bigoplus_{j=1}^k\mathbb{Z}/n_j\mathbb{Z}$ where $n_1,\ldots,n_k\in\mathbb{Z}_{\geq1}$ (i.e. \!$\prod_{j=1}^kn_j=|A|$ and $n_1\mid n_2\mid\cdots\mid n_k$), then \cite[Theorem 1]{MR2963406} states that
\begin{equation}\label{eq:point}
|\mathrm{Orb}(\mathcal{C}(A,q))|=\sum_{d_1\mid n_1,\ldots,d_k\mid n_k}\dfrac{\phi(d_1)\cdots\phi(d_k)}{\phi(\mathrm{lcm}(d_1,\ldots,d_k))}.
\end{equation}
One important observation is that the Galois action on $\mathcal{C}(A,q)$ depends only on the group structure of the finite abelian group $A$, and not on the nondegenerate quadratic form $q$, hence even rough characteristics of modular tensor categories cannot be derived from the Galois action.  For example,  $\mathcal{C}(\mathfrak{sl}_2,1)^{\boxtimes2}$ and $\mathcal{C}(\mathfrak{so}_8,1)$ are both pointed with underlying abelian group $\mathbb{Z}/2\mathbb{Z}\oplus\mathbb{Z}/2\mathbb{Z}$.  The Galois action on both categories is trivial, but the latter category is prime while the other factors as a Deligne tensor product.  We illustrate the diversity of the number of Galois orbits of simple objects in pointed modular tensor categories of a fixed rank by example in Figure \ref{fig:point}.
\begin{figure}[H]
\centering
\begin{equation*}
\begin{array}{|c|c||c|c|}
\hline A & |\mathrm{Orb}(\mathcal{C}(A,q))| & A & |\mathrm{Orb}(\mathcal{C}(A,q))| \\\hline\hline
\mathbb{Z}/2\mathbb{Z}\oplus\mathbb{Z}/30\mathbb{Z}\oplus\mathbb{Z}/30\mathbb{Z} & 280 &\mathbb{Z}/10\mathbb{Z}\oplus\mathbb{Z}/180\mathbb{Z} & 126\\
\mathbb{Z}/2\mathbb{Z}\oplus\mathbb{Z}/6\mathbb{Z}\oplus\mathbb{Z}/150\mathbb{Z} & 120 &\mathbb{Z}/2\mathbb{Z}\oplus\mathbb{Z}/900\mathbb{Z} & 54\\
\mathbb{Z}/2\mathbb{Z}\oplus\mathbb{Z}/10\mathbb{Z}\oplus\mathbb{Z}/90\mathbb{Z} &168 &\mathbb{Z}/15\mathbb{Z}\oplus\mathbb{Z}/120\mathbb{Z} & 140\\
\mathbb{Z}/2\mathbb{Z}\oplus\mathbb{Z}/2\mathbb{Z}\oplus\mathbb{Z}/450\mathbb{Z} & 72 &\mathbb{Z}/3\mathbb{Z}\oplus\mathbb{Z}/600\mathbb{Z} & 60 \\
\mathbb{Z}/30\mathbb{Z}\oplus\mathbb{Z}/60\mathbb{Z} &210 &\mathbb{Z}/5\mathbb{Z}\oplus\mathbb{Z}/360\mathbb{Z} & 84\\
\mathbb{Z}/6\mathbb{Z}\oplus\mathbb{Z}/300\mathbb{Z} & 90&\mathbb{Z}/1800\mathbb{Z} & 36 \\\hline
\end{array}
\end{equation*}
    \caption{Number of Galois orbits of pointed modular tensor categories of rank 1800}%
    \label{fig:point}%
\end{figure}
\end{example}

\begin{note}[Cyclic groups]
Let $n\in\mathbb{Z}_{\geq1}$.  Equation (\ref{eq:point}) states
\begin{equation}
|\mathrm{Orb}(\mathcal{C}(\mathbb{Z}/n\mathbb{Z},q))|=\sum_{d\in\mathbb{Z}_{\geq1}, d\mid n}\dfrac{\phi(d)}{\phi(\mathrm{lcm}(d))}=\sum_{d\in\mathbb{Z}_{\geq1}, d\mid n}1,
\end{equation}
the number of positive integer divisors of $n$.
\end{note}

\begin{note}[Elementary abelian $p$-groups]
Let $p\in\mathbb{Z}_{\geq2}$ be prime, $n\in\mathbb{Z}_{\geq1}$, and $A=(\mathbb{Z}/p\mathbb{Z})^{\oplus n}$.  Equation (\ref{eq:point}) states
\begin{equation}
|\mathrm{Orb}(\mathcal{C}(A,q))|=\sum_{d_1\mid p,\ldots,d_n\mid p}\dfrac{\prod_{j=1}^n\phi(d_j)}{\phi(\mathrm{lcm}(d_1,\ldots,d_n))}=1+\sum_{k=1}^n\binom{n}{k}(p-1)^{k-1}=1+\dfrac{p^n-1}{p-1}.
\end{equation}
\end{note}


\begin{example}[Deligne products of transitive categories]\label{transdeligne}
If two transitive modular tensor categories are coprime in the sense that their conductors are coprime integers, then their Deligne product is transitive as well by \cite[Lemma 2.1(iii)]{2020arXiv200701366N}.  Now consider a product $\mathcal{T}_1\boxtimes\mathcal{T}_2$ of transitive modular tensor categories.  Each of $\mathcal{T}_1$ and $\mathcal{T}_2$ has a unique factorization up to reordering into coprime simple transitive factors \cite[Theorem II]{2020arXiv200701366N}.  So it is clear that by a organizing these simple factors by their associated prime conductor, one can compute $|\mathrm{Orb}(\mathcal{T}_1\boxtimes\mathcal{T}_2)|$ with only the knowledge of $|\mathrm{Orb}(\mathcal{T}^{\boxtimes2})|$ for an abitrary transitive modular tensor category $\mathcal{T}$.  To this end, recall that simple objects $X\in\mathcal{T}$ are classified by a Galois automorphism $\sigma\in\mathrm{Gal}(\mathbb{Q}(\dim(\mathcal{C}))/\mathbb{Q})$ such that $\hat{\sigma}(\mathbbm{1})=X$ \cite[Section 3]{2020arXiv200701366N}.  We will henceforth use the bijection $\mathcal{O}(\mathcal{T})=\{X_\sigma:\sigma\in\mathrm{Gal}(\mathbb{Q}(\dim(\mathcal{C}))/\mathbb{Q})\}$.  Note that the $\mathrm{rank}(\mathcal{T})$ Galois orbits of $\mathcal{T}^{\boxtimes2}$, $\{\mathcal{O}_{\mathbbm{1}\boxtimes X_\sigma}:\sigma\in\mathrm{Gal}(\mathbb{Q}(\dim(\mathcal{C}))/\mathbb{Q})\}$ are pairwise disjoint since the Galois action on $\mathcal{O}(\mathcal{T})$ is fixed-point free \cite[Proposition 3.2]{2020arXiv200701366N}.  And for the same reason, for each $\sigma$, $|\mathcal{O}_{\mathbbm{1}\boxtimes X_\sigma}|=\mathrm{rank}(\mathcal{T})$.  Hence this set contains all Galois orbits of $\mathcal{T}^{\boxtimes2}$ and in particular, $|\mathrm{Orb}(\mathcal{T}^{\boxtimes2})|=\mathrm{rank}(\mathcal{T})$.  This result also follows from \cite[Proposition 3.12]{2020arXiv200701366N}.
\end{example}

In the following examples the simple objects will be labeled/ordered as $\mathbbm{1}=X_0,X_1,X_2,\ldots$.

\begin{example}\label{fib}
Let $\mathrm{Fib}$ be any rank 2 modular tensor category with nontrivial simple object of dimension $u:=(1/2)(1+\sqrt{5})$ \cite{ostrik}.  Let $\sigma$ be the nontrivial element of $\mathrm{Gal}(\mathbb{Q}(\zeta_5)^+/\mathbb{Q})$.  There are two distinct Deligne products up to Galois conjugacy and equivalence, of the form $\mathrm{Fib}^{\tau_1}\boxtimes\mathrm{Fib}^{\tau_2}$ where $\tau_1,\tau_2$ are either $\mathrm{id}$ or $\sigma$.  Their $s$-matrices are
\begin{equation}
s = \left[\begin{array}{cc|cc}
1 & u^2 & u & u \\
u^2 & 1 & -u & -u \\\hline
u & -u & -1 & u^2 \\
u & -u & u^2 & -1
\end{array}\right]
\qquad\text{ and }\qquad
s = \left[\begin{array}{cc|cc}
1 & -1 & u & \sigma(u) \\
-1 & 1 & -\sigma(u) & -u \\\hline
u & -\sigma(u) & -1 & -1 \\
\sigma(u) & -u & -1 & -1
\end{array}\right].
\end{equation}
One can easily verify that $\{X_0,X_1\}$ and $\{X_2,X_3\}$ are the orbits of the Galois action.  But $\mathrm{Fib}\boxtimes\mathrm{Fib}$ is pseudounitary while $\mathrm{Fib}\boxtimes\mathrm{Fib}^\sigma$ is not pseudounitary, nor Galois conjugate to a pseudounitary category.
\end{example}

\begin{example}
Consider the rank 6 non-psuedounitary modular tensor category $\mathcal{C}(\mathfrak{so}_5,3/2)_\mathrm{ad}$ \cite[Section 5.2]{rowell}.  The entries of the $s$-matrix lie in the field $\mathbb{Q}(\zeta_9)^+$ where $\zeta_n:=\exp(2\pi i/n)$ for $n\in\mathbb{Z}_{\geq1}$, the maximal totally real subfield of the ninth roots of unity.  Define $u:=\zeta_9-\zeta_9^2-\zeta_9^5$ and $\sigma\in\mathrm{Gal}(\mathbb{Q}(\zeta_9)^+/\mathbb{Q})$ such that $\sigma(\zeta_9)=\zeta_9^2$.  Then

\begin{equation}
s=\left[\begin{array}{ccc|ccc}
1 & -1 & 1 & u & \sigma(u) & \sigma^2(u) \\
-1 & 1 & -1 & -\sigma(u) & -\sigma^2(u) & -u \\
1 & -1 & 1 & \sigma^2(u) & u & \sigma(u) \\\hline
u & -\sigma(u) & \sigma^2(u) & 1 & 1 & 1 \\
\sigma(u) &-\sigma^2(u) & u & 1 & 1 & 1 \\
\sigma^2(u) &-u & \sigma(u) & 1 & 1 & 1
\end{array}\right].
\end{equation}
One easily computes that $\{X_0,X_1,X_2\}$ and $\{X_3,X_4,X_5\}$ are the orbits of the Galois action.
\end{example}

\begin{example}
Consider the rank 5 unitary modular tensor category $\mathcal{C}(\mathfrak{sl}_2,12)_A^0$ \cite[Section 7]{KiO}.  The modular data lies in $\mathbb{Q}(\zeta_7)$.  We compute
\begin{equation}
s = \left[\begin{array}{ccc|cc}
1 & a & b & 1-a+b & 1-a+b \\
a & b & -1 & -1+a-b & -1+a-b \\
b & -1 & -a & 1-a+b & 1-a+b \\\hline
1-a+b & -1+a-b & 1-a+b & c & \overline{c} \\
1-a+b & -1+a-b & 1-a+b & \overline{c} & c
\end{array}\right],
\end{equation}
where $a:=-\zeta_7^4-\zeta_7^3+1$, $b:=-\zeta_7^5-2\zeta_7^4-2\zeta_7^3-\zeta_7^2$, and $c:=-\zeta_7^3-\zeta_7^2-2\zeta_7-1$.    This category is not self-dual, unlike the previous examples as $c\not\in\mathbb{R}$.  One easily computes that $\{X_0,X_1,X_2\}$ and $\{X_3,X_4\}$ are the orbits of the Galois action.
\end{example}


\section{Generalities}\label{sec:gen}


\subsection{Galois conjugacy, centralizers, and fusion subcategories}

Here we prove Theorem \ref{prelem0}, a classification of fusion subcategories of modular tensor categories invariant under the Galois action, which generalizes \cite[Lemma 3.3(ii)]{MR3632091} to arbitrary fusion subcategories.  This implies Theorem \ref{what}, a lower bound for the number of Galois orbits of a modular tensor category based on the number of invertible objects it possesses, i.e. the order of its universal grading group.  Recall that the conductor of a modular tensor category is the order of its $t$-matrix.

\begin{lemma}\label{prelemminus1}
Let $\mathcal{C}$ be a modular tensor category with conductor $N\in\mathbb{Z}_{\geq1}$, and  $X,Y\in\mathcal{O}(\mathcal{C})$.  Then $X$ and $Y$ centralize one another if and only if $X$ and $\hat{\sigma}(Y)$ centralize one another for all $\sigma\in\mathrm{Gal}(\mathbb{Q}(\zeta_N)/\mathbb{Q}(\dim(X)))$.
\end{lemma}

\begin{proof}
The converse direction is trivial.  For the forward direction, the definition of the Galois action in Equation  (\ref{gal}) states that for all $\sigma\in\mathrm{Gal}(\mathbb{Q}(\zeta_N)/\mathbb{Q}(\dim(X)))\subset\mathrm{Gal}(\mathbb{Q}(\zeta_N)/\mathbb{Q})$, we have $s_{X,\hat{\sigma}(Y)}=\sigma(s_{X,Y}/s_{\mathbbm{1},Y})\dim(\hat{\sigma}(Y))$.  But $s_{X,Y}/s_{\mathbbm{1},Y}=\dim(X)$ as $X$ and $Y$ centralize one another, thus $s_{X,\hat{\sigma}(Y)}=\sigma(\dim(X))\dim(\hat{\sigma}(Y))=\dim(X)\dim(\hat{\sigma}(Y))$.
\end{proof}

\begin{lemma}\label{counting}
If $\mathcal{C}$ is a modular tensor category, then $\mathcal{O}_\mathbbm{1}\subset\mathcal{O}(\mathcal{C}_\mathrm{ad})$.  Moreover if $\mathcal{O}(\mathcal{C}_\mathrm{pt})\subset\mathcal{O}_\mathbbm{1}$, then $\mathcal{C}_\mathrm{pt}$ is symmetrically braided. 
\end{lemma}

\begin{proof}
The first claim follows from Lemma \ref{prelemminus1} because $\mathbbm{1}\in\mathcal{O}(\mathcal{C}_\mathrm{ad})$.  In particular $\mathcal{O}_\mathbbm{1}$ centralizes $\mathcal{C}_\mathrm{pt}$, hence $\mathcal{C}_\mathrm{pt}$ centralizes itself when $\mathcal{O}(\mathcal{C}_\mathrm{pt})\subset\mathcal{O}_\mathbbm{1}$.
\end{proof}

Recall the definition of the fields of Verlinde eigenvalues $\mathbb{L}_X$ from Section \ref{secgal}.

\begin{lemma}\label{counting2}
Let $\mathcal{C}$ be a modular tensor category and $\mathcal{D}\subset\mathcal{C}$ a fusion subcategory.  Define the field of dimensions $\mathbb{K}_\mathcal{D}:=\mathbb{Q}(\dim(Y):Y\in\mathcal{O}(C_\mathcal{C}(\mathcal{D})))$.  If $X\in\mathcal{O}(\mathcal{D})$, then
\begin{equation}
|\mathcal{O}_X\cap\mathcal{O}(\mathcal{D})|=[\mathbb{L}_X:\mathbb{K}_\mathcal{D}\cap\mathbb{L}_X]=\dfrac{|\mathcal{O}_X|}{[\mathbb{K}_\mathcal{D}\cap\mathbb{L}_X:\mathbb{Q}]}.
\end{equation}
\end{lemma}

\begin{proof}
By definition, any $X\in\mathcal{O}(\mathcal{D})$ is centralized by all $Y\in\mathcal{O}(C_\mathcal{C}(\mathcal{D}))$ which by Lemma \ref{prelemminus1} occurs if and only if $\hat{\sigma}(X)$ is centralized by all $Y\in\mathcal{O}(C_\mathcal{C}(\mathcal{D}))$ for all $\sigma\in\mathrm{Gal}(\mathbb{Q}(\zeta_N)/\mathbb{Q}(\dim(Y)))$ where $N\in\mathbb{Z}_{\geq1}$ is the conductor of $\mathcal{C}$.  Equivalenty, $\hat{\sigma}(X)\in \mathcal{O}(C_\mathcal{C}(C_\mathcal{C}(\mathcal{D})))=\mathcal{O}(\mathcal{D})$ \cite[Theorem 8.21.1(ii)]{tcat}.  Therefore $X\in\mathcal{O}(\mathcal{D})$ if and only if $\hat{\sigma}(X)\in\mathcal{O}(\mathcal{D})$ for all
\begin{equation}
\sigma\in\bigcap_{Y\in\mathcal{O}(C_\mathcal{C}(\mathcal{D}))}\mathrm{Gal}(\mathbb{Q}(\zeta_N)/\mathbb{Q}(\dim(Y)))=\mathrm{Gal}(\mathbb{Q}(\zeta_N)/\mathbb{K}_\mathcal{D}).
\end{equation}
As $\mathcal{O}_X$ is in bijection with $\mathrm{Gal}(\mathbb{L}_X/\mathbb{Q})$ by Lemma \ref{lem1}, we have $|\mathcal{O}_X\cap\mathcal{O}(\mathcal{D})|=[\mathbb{L}_X:\mathbb{K}_\mathcal{D}\cap\mathbb{L}_X]$.
\end{proof}

\begin{example}
It is instructive to apply Lemma \ref{counting2} in extreme cases.  For example, consider $\mathbbm{1}\in\mathcal{O}(\mathcal{C}_\mathrm{pt})$.  Lemma \ref{counting2} implies $|\mathcal{O}_\mathbbm{1}|/|\mathcal{O}_\mathbbm{1}\cap\mathcal{O}(\mathcal{C}_\mathrm{pt})|=[\mathbb{Q}(\dim(Y):Y\in\mathcal{O}(\mathcal{C}_\mathrm{ad})):\mathbb{Q}]$ since $\mathbb{Q}(\dim(Y):Y\in\mathcal{O}(\mathcal{C}_\mathrm{ad}))\subset\mathbb{L}_\mathbbm{1}$.  But $|\mathcal{O}_\mathbbm{1}|=[\mathbb{L}_\mathbbm{1}:\mathbb{Q}]$, hence $|\mathcal{O}_\mathbbm{1}\cap\mathcal{O}(\mathcal{C}_\mathrm{pt})|=[\mathbb{L}_\mathbbm{1}:\mathbb{Q}(\dim(Y):Y\in\mathcal{O}(\mathcal{C}_\mathrm{ad}))]$.  In \cite{2019arXiv191212260G}, a dimensional grading was defined for all fusion rings from the number fields generated by Frobenius-Perron dimensions \cite[Proposition 1.8]{2019arXiv191212260G}.  It is clear this grading group is isomorphic to $\mathcal{O}_\mathbbm{1}\cap\mathcal{O}(\mathcal{C}_\mathrm{pt})$ when $\mathcal{C}$ is a pseudounitary modular tensor category.
\begin{question}
Does a (nonpseudounitary) modular tensor category $\mathcal{C}$ exist such that the dimensional grading group from \textnormal{\cite{2019arXiv191212260G}} is not isomorphic to $\mathcal{O}_\mathbbm{1}\cap\mathcal{O}(\mathcal{C}_\mathrm{pt})$?
\end{question}
\end{example}

\begin{theorem}\label{prelem0}
Let $\mathcal{C}$ be a modular tensor category and $\mathcal{D}\subset\mathcal{C}$ a fusion subcategory.  Then $\mathcal{D}$ is closed under the Galois action of $\mathcal{C}$ if and only if $C_\mathcal{C}(\mathcal{D})$ is integral. 
\end{theorem}

\begin{proof}
\par The converse direction follows from Lemma \ref{counting2} applied to all $X\in\mathcal{O}(\mathcal{D})$.  Now assume $\mathcal{D}$ is closed under the Galois action of $\mathcal{C}$.  Fix $X\in\mathcal{O}(C_\mathcal{C}(\mathcal{D}))$ and let $\sigma\in\mathrm{Gal}(\mathbb{Q}(\zeta_N)/\mathbb{Q})$ be arbitrary where $N\in\mathbb{Z}_{\geq1}$ is the conductor of $\mathcal{C}$.  The definition of the Galois action (\ref{gal}) states that for all $Y\in\mathcal{O}(\mathcal{D})$,
\begin{equation}
\dim(X)\dim(\hat{\sigma}(Y))=s_{X,\hat{\sigma}(Y)}=\sigma\left(\dfrac{s_{X,Y}}{s_{\mathbbm{1},Y}}\right)\dim(\hat{\sigma}(Y))=\sigma(\dim(X))\dim(\hat{\sigma}(Y)).
\end{equation}
where the first equality follows from $\mathcal{D}$ being closed under the Galois action of $\mathcal{C}$ and the last equality follows from $X$ and $Y$ centralizing one another.   As $\dim(\hat{\sigma}(Y))\neq0$ \cite[Theorem 2.3]{ENO}, then $\dim(X)$ is fixed by all such $\sigma$, and therefore $\dim(X)\in\mathbb{Z}$ for all $X\in\mathcal{O}(C_\mathcal{C}(\mathcal{D}))$.  But $C_\mathcal{C}(\mathcal{D})$ is spherical, so from the proof of \cite[Proposition 8.22]{ENO}, $\dim(C_\mathcal{C}(\mathcal{D}))/\mathrm{FPdim}(C_\mathcal{C}(\mathcal{D}))$ is equal to the dimension of a simple object in the Drinfeld center $\mathcal{Z}(C_\mathcal{C}(\mathcal{D}))$, a modular tensor category.  Noting that the forgetful functor $F:\mathcal{Z}(C_\mathcal{C}(\mathcal{D}))\to C_\mathcal{C}(\mathcal{D})$ preserves dimension, we have $\dim(C_\mathcal{C}(\mathcal{D}))/\mathrm{FPdim}(C_\mathcal{C}(\mathcal{D}))$ is a positive integer $\leq1$ \cite[Proposition 8.22]{ENO}, thus $\dim(C_\mathcal{C}(\mathcal{D}))=\mathrm{FPdim}(C_\mathcal{C}(\mathcal{D}))$ .  In other words $C_\mathcal{C}(\mathcal{D})$ is pseudounitary, and therefore $C_\mathcal{C}(\mathcal{D})$ is integral by \cite[Lemma 3.3(i)]{MR3632091}. 
\end{proof}

\begin{note}
Recall from \cite[Theorem 1]{Sawin06} that if $k\neq2$, the fusion subcategories of $\mathcal{C}:=\mathcal{C}(\mathfrak{g},k)$ are all of the form $\mathcal{D}\subset\mathcal{C}_\mathrm{pt}$ or $C_\mathcal{C}(\mathcal{D})$.  Moreover if $k\neq2$, all fusion subcategories of $\mathcal{C}(\mathfrak{g},k)$ which are not pointed, are closed under the Galois action on $\mathcal{O}(\mathcal{C})$.  This is also clear from geometric reasoning \cite[Section 4]{gannoncoste}.
\end{note}

\begin{corollary}\label{precor}
If $\mathcal{C}$ is a modular tensor category, then $\mathcal{C}_\mathrm{ad}$ is closed under the Galois action.
\end{corollary}

\begin{proof}
This follows from Theorem \ref{prelem0} as $C_\mathcal{C}(\mathcal{C}_\mathrm{ad})=\mathcal{C}_\mathrm{pt}$ \cite[Corollary 6.9]{nilgelaki}.
\end{proof}

\begin{theorem}\label{what}
Let $\mathcal{C}$ be a modular tensor category such that $\dim(\mathcal{C}_\mathrm{pt})=\mathrm{rank}(\mathcal{C}_\mathrm{pt})=\prod_{j\in J}p_j^{a_j}$ where $p_j$ are distinct primes and $a_j\in\mathbb{Z}_{\geq1}$,  indexed by a finite set $J$.  Then $|\mathrm{Orb}(\mathcal{C})|\geq1+\sum_{j\in J}a_j$.
\end{theorem}

\begin{proof}
By the fundamental theorem of finite abelian groups we have a composition series of $\mathcal{O}(\mathcal{C}_\mathrm{pt})$ as an abelian group of length $n:=\sum_{j\in J}a_j$, corresponding to fusion subcategories
\begin{equation}
\mathrm{Vec}=\mathcal{P}_0\subsetneq\mathcal{P}_1\subsetneq\cdots\subsetneq\mathcal{P}_{n-1}\subset\mathcal{P}_n=\mathcal{C}_\mathrm{pt}.
\end{equation} 
The double-centralizer property \cite[Theorem 8.21.1(ii)]{tcat} implies that for each $0\leq j\leq n-1$, we have $C_\mathcal{C}(\mathcal{P}_{j+1})\subsetneq C_\mathcal{C}(\mathcal{P}_j)$ is a strict inclusion of fusion subcategories.  By Theorem \ref{prelem0}, $C_\mathcal{C}(\mathcal{P}_j)$ is closed under Galois conjugacy for all $0\leq j\leq n$.  Therefore $\mathcal{C}=\mathcal{P}_n$ has greater than or equal to $n+1$ Galois orbits of simple objects.
\end{proof}


\subsection{Pseudoinvertible and pseudounitary}

Recall that a pseudoinvertible object $X$ is one such that $\dim(X)=\pm1$, and a pseudounitary fusion category is one such that $\dim(\mathcal{C})=\mathrm{FPdim}(\mathcal{C})$.

\begin{lemma}\label{pseudo}
Let $\mathcal{C}$ be a modular tensor category and $\Gamma\subset\mathcal{O}(\mathcal{C})$ be the union of Galois orbits of $\mathcal{C}$ containing a pseudoinvertible object.  If there exists $X\in\mathcal{O}(\mathcal{C})$ with $\dim(X)=\mathrm{FPdim}(X)$ such that all $Y\in\mathcal{O}(\mathcal{C})$ which centralize $X$ belong to $\Gamma$, then $\mathcal{C}$ is Galois conjugate to a pseudounitary category.
\end{lemma}

\begin{proof}
Each character of the Grothendieck ring $K(\mathcal{C})$ is indexed by $Y\in\mathcal{O}(\mathcal{C})$; the corresponding character is $\varphi_Y:K(\mathcal{C})\to\mathbb{C}$ such that $Z\mapsto s_{Z,Y}/\dim(Y)$.    But for any $Y\in\mathcal{O}(\mathcal{C})\setminus\Gamma$, $s_{X,Y}\neq\dim(X)\dim(Y)$ because $X$ and $Y$ do not centralize one another by assumption.  Hence $\varphi_Y(X)\neq\dim(X)$ for all $Y\in\mathcal{O}(\mathcal{C})\setminus\Gamma$.  But $\mathrm{FPdim}(X)=\dim(X)$, hence as a character of $K(\mathcal{C})$, $\mathrm{FPdim}=\varphi_Z$ for some $Z\in\Gamma$.  Specifically, let $\sigma\in\mathrm{Gal}(\mathbb{Q}(\zeta_N)/\mathbb{Q})$ such that $\hat{\sigma}(W)=Z$ where $W$ is pseudoinvertible.  Then we have $\mathrm{FPdim}(\mathcal{C})=\dim(\mathcal{C})/\dim(\hat{\sigma}(W))^2=\sigma(\dim(\mathcal{C}))=\dim(\mathcal{C}^\sigma)$.  The category $\mathcal{C}^\sigma$ is therefore pseudounitary.
\end{proof}

\begin{note}
Recall the category $\mathcal{C}:=\mathrm{Fib}\boxtimes\mathrm{Fib}^\sigma$ from Example \ref{fib}.  The category $\mathcal{C}$ possesses both nontrivial pseudoinvertible objects and a nontrivial simple object $X_1$ with $\mathrm{FPdim}(X_1)=\dim(X_1)=(1/2)(1+\sqrt{5})$.  But $X_1$ is centralized by its only other Galois conjugate, which is not pseudoinvertible.  Thus Lemma \ref{pseudo} fails.  Indeed $\mathcal{C}$ is not Galois conjugate to a pseudounitary category.    
\end{note}

\begin{theorem}\label{what2}
Let $\mathcal{C}$ be a modular tensor category.   Then every Galois orbit of $\mathcal{C}$ contains a pseudoinvertible object if and only if $\mathcal{C}\simeq\mathcal{C}_\mathrm{pt}\boxtimes\mathcal{T}$ where $\mathcal{T}$ is a transitive modular tensor category.
\end{theorem}

\begin{proof}
For the forward direction, we may assume that $\mathcal{C}$ is pseudounitary by applying Lemma \ref{pseudo} with $X=\mathbbm{1}$.  Thus there exists at least one invertible object in each Galois orbit of $\mathcal{C}$.   Lemma \ref{prelemminus1} then implies any $Y\in\mathcal{O}(\mathcal{C}_\mathrm{pt}')$ centralizes $\mathcal{O}(\mathcal{C})$, hence $Y=\mathbbm{1}$, $\mathcal{C}_\mathrm{pt}'=\mathrm{Vec}$, and moreover $\mathcal{C}\simeq\mathcal{C}_\mathrm{pt}\boxtimes\mathcal{C}_\mathrm{ad}$ by \cite[Proposition 2.2]{DMNO}.  Lemma \ref{counting} states $\mathcal{O}_\mathbbm{1}\subset\mathcal{O}(\mathcal{C}_\mathrm{ad})$.  But every Galois orbit contains an invertible object and $\mathcal{O}(\mathcal{C}_\mathrm{ad})\cap\mathcal{O}(\mathcal{C}_\mathrm{pt})=\{\mathbbm{1}\}$, so $\mathcal{O}(\mathcal{C}_\mathrm{ad})=\mathcal{O}_\mathbbm{1}$ is transitive.  The converse implication follows from \cite[Lemma 2.1(ii)]{2020arXiv200701366N} which implies that each $g\boxtimes X\in\mathcal{O}(\mathcal{C}_\mathrm{pt}\boxtimes\mathcal{T})$ is Galois conjugate to $h\boxtimes\mathbbm{1}$ for some $h\in\mathcal{O}(\mathcal{C}_\mathrm{pt})$ since $\mathcal{T}$ is transitive, which is an invertible object, hence pseudoinvertible.
\end{proof}

\begin{note}
Theorem \ref{what2} can be restated in terms of formal codegrees \cite{MR2576705} as follows.  Let $\mathcal{C}$ be a modular tensor category.  Then the set of formal codegrees of $\mathcal{C}$ are a single Galois orbit if and only if $\mathcal{C}\simeq\mathcal{C}_\mathrm{pt}\boxtimes\mathcal{T}$ where $\mathcal{T}$ is a transitive modular tensor category.
\end{note}

\begin{question}
Does a fusion category $\mathcal{C}$ exist (not necessarily braided or modular) whose formal codegrees are all Galois conjugate but $\mathcal{C}\not\simeq\mathcal{C}_\mathrm{pt}\boxtimes\mathcal{C}_\mathrm{ad}$ as fusion categories?
\end{question}

We end this section by computing the number of Galois orbits of the modular tensor categories characterized by Theorem \ref{what2}.

\begin{example}[Product of pointed and transitive]
Pointed modular tensor categories factor as a Deligne product of coprime factors \cite[Theorem 1.1]{drinfeld2007grouptheoretical} whose $t$-matrix has prime power order \cite[Theorem 3.9]{paul}, while transitive modular tensor categories factor as a Deligne product of coprime factors whose $t$-matrix has prime order greater than or equal to 5 \cite[Theorem II]{2020arXiv200701366N}.  So fix a prime $p\in\mathbb{Z}_{\geq5}$.  To compute the number of Galois orbits of an arbitrary Deligne product $\mathcal{C}\boxtimes\mathcal{T}$ where $\mathcal{C}$ is pointed and $\mathcal{T}$ is transitive, it suffices to compute $|\mathrm{Orb}(\mathcal{C}(A,q)\boxtimes\mathcal{C}(\mathfrak{sl}_2,p-2)_\mathrm{ad})|$ where $A$ is a finite abelian group of order $p^n$ for some $n\in\mathbb{Z}_{\geq1}$ and $q$ is any nondegenerate quadratic form.  We assume $\mathcal{C}$ and $\mathcal{T}$ are of this form henceforth and denote their tensor units by $\mathbbm{1}_\mathcal{C}$ and $\mathbbm{1}_\mathcal{T}$, respectively.
\par First note that every $g\boxtimes X\in\mathcal{O}(\mathcal{C}\boxtimes\mathcal{T})$ is Galois conjugate to an element of the form $h\boxtimes\mathbbm{1}_\mathcal{T}$ for some $h\in\mathcal{O}(\mathcal{C})$ since $\mathcal{T}$ is transitive.  In the case $g=\mathbbm{1}_\mathcal{C}$, $|\mathcal{O}_{\mathbbm{1}_\mathcal{C}\boxtimes\mathbbm{1}_\mathcal{T}}|=(p-1)/2$ since $\mathbbm{1}_\mathcal{C}$ is fixed by the Galois action of $\mathcal{C}$.  Now it suffices to count the number of Galois congugacy classes of elements of the form $g\boxtimes\mathbbm{1}_\mathcal{T}$ for nontrivial $g\in\mathcal{O}(\mathcal{C})$.  Let $g\in\mathcal{O}(\mathcal{C})$ such that the order of $g$ is $p^k$ for some $k\in\mathbb{Z}_{\geq1}$.  Then $\hat{\sigma}(g\boxtimes\mathbbm{1}_\mathcal{T})=h\boxtimes\mathbbm{1}_\mathcal{T}$ if and only if $\hat{\sigma}(g)=h$ under the Galois action of $\mathcal{C}$ and $\hat{\sigma}(\mathbbm{1}_\mathcal{T})=\mathbbm{1}_\mathcal{T}$ under the Galois action of $\mathcal{T}$.  As was shown in Example \ref{point}, the Galois orbit of $g\in\mathcal{O}(\mathcal{C})$ is determined by $\sigma\in\mathrm{Gal}(\mathbb{Q}(\zeta_{p^k})/\mathbb{Q})$ and the condition $\hat{\sigma}(\mathbbm{1}_\mathcal{T})=\mathbbm{1}_\mathcal{T}$ implies $\sigma(\dim(\mathcal{T}))=\dim(\mathcal{T})$, i.e. $\sigma\in\mathrm{Gal}(\mathbb{Q}(\zeta_{p^k})/\mathbb{Q}(\zeta_p)^+)$.  Thus $g\boxtimes\mathbbm{1}_\mathcal{T}$ has exactly $|\mathrm{Gal}(\mathbb{Q}(\zeta_{p^k})/\mathbb{Q}(\zeta_p)^+)|=\phi(p^k)/(\phi(p)/2)=2p^{k-1}$ Galois conjugates of the form $h\boxtimes\mathbbm{1}_\mathcal{T}$.
\par Therefore, to each Galois orbit of an element $g\in\mathcal{O}(\mathcal{C})$ of order $p^k$, there corresponds $\phi(p^k)/(2p^{k-1})=(p-1)/2$ Galois orbits of $\mathcal{C}\boxtimes\mathcal{T}$, independent of $k$.  We may then sum over all Galois orbits of $\mathcal{C}$, ommitting $\mathcal{O}_\mathbbm{1}$, to yield
\begin{equation}
|\mathrm{Orb}(\mathcal{C}\boxtimes\mathcal{T})|=1+(|\mathrm{Orb}(\mathcal{C})|-1)\dfrac{p-1}{2}.
\end{equation}

\begin{note}[Cyclic groups]\label{cyclic}
If $p\in\mathbb{Z}_{\geq5}$ is prime and $n\in\mathbb{Z}_{\geq1}$, then
\begin{equation}
|\mathrm{Orb}(\mathcal{C}(\mathbb{Z}/p^n\mathbb{Z},q)\boxtimes\mathcal{C}(\mathfrak{sl}_2,p-2)_\mathrm{ad})|=1+\left((n+1)-1\right)\dfrac{(p-1)}{2}=\dfrac{n(p-1)+2}{2}.
\end{equation}
\end{note}

\begin{note}[Elementary abelian $p$-groups]
If $p\in\mathbb{Z}_{\geq5}$ is prime and $n\in\mathbb{Z}_{\geq1}$, then
\begin{equation}
|\mathrm{Orb}(\mathcal{C}((\mathbb{Z}/p\mathbb{Z})^{\oplus n},q)\boxtimes\mathcal{C}(\mathfrak{sl}_2,p-2)_\mathrm{ad})|=1+\left(1+\dfrac{p^n-1}{p-1}-1\right)\dfrac{(p-1)}{2}=\dfrac{p^n+1}{2}.
\end{equation}
\end{note}
\end{example}



\section{Two orbits}\label{sec:two}

The general results of Section \ref{sec:gen} can be used to classify modular tensor categories with a small number of Galois orbits of simple objects and nontrivial universal grading.  We provide a small table of examples in Figure \ref{fig:two}.  Here we reduce the classification of modular tensor categories $\mathcal{C}$ with $|\mathrm{Orb}(\mathcal{C})|=2$ to the case where $\mathcal{C}$ is simple, i.e.\ $\mathcal{C}$ does not possess any nontrivial fusion subcategories.  The culiminating results are Corollary \ref{pointedprop} which describes the case $\mathcal{C}_\mathrm{pt}$ is nontrivial, and Proposition \ref{unpointedprop} which describes the case when $\mathcal{C}_\mathrm{pt}$ is trivial.  With increasing complexity, these results can be extended to the cases when $|\mathrm{Orb}(\mathcal{C})|>2$.  

\begin{figure}[H]
\centering
\begin{equation*}
\begin{array}{|c|c|c|c|}
\hline \mathcal{C} & \mathrm{rank}(\mathcal{C}) & \mathrm{rank}(\mathcal{C}_\mathrm{pt}) &\text{Orbit sizes} \\\hline\hline
\mathcal{C}(\mathfrak{sl}_2,2) & 3 & 2 & 1,2 \\\hline
\mathcal{C}(\mathfrak{g}_2,-2/3) & 4 & 1 & 2,2 \\\hline
\mathcal{C}(\mathfrak{so}_5,3/2)_\mathrm{ad} & 6 & 1 & 3,3 \\\hline
\mathcal{C}(\mathfrak{so}_5,5/2)_\mathrm{ad} & 10 & 1 & 5,5 \\\hline
\mathcal{C}(\mathfrak{g}_2,5) & 12 & 1 & 3,9 \\\hline
\mathcal{C}(\mathfrak{g}_2,5/3) & 16 & 1 & 8,8 
\\\hline
\mathcal{C}(E_7,5)_\mathrm{ad} & 22 & 1 & 11,11
\\\hline\hline
\mathcal{C}(\mathfrak{sl}_p,1) & p & p & 1,\phi(p) \\
p\geq2\text{ prime}& & & \\\hline
\mathcal{C}(\mathfrak{sl}_2,2(p-1))_A^0 & (p+3)/2 & 1 & 2,\frac{1}{2}\phi(p) \\
p\geq5\text{ prime}& & & \\\hline
\mathcal{C}(\mathfrak{sl}_2,p^2-2)_\mathrm{ad}  & (p^2-1)/2  & 1 & \frac{1}{2}\phi(p),\frac{1}{2}\phi(p^2) \\
p\geq3\text{ prime}& & &  \\\hline
\end{array}
\end{equation*}
    \caption{Examples of modular tensor categories $\mathcal{C}$ with $|\mathrm{Orb}(\mathcal{C})|=2$}%
    \label{fig:two}%
\end{figure}


\subsection{Possible pointed subcategories}

The goal of this subsection is to prove that when a modular tensor category $\mathcal{C}$ has two Galois orbits of simple objects and nontrivial pointed subcategory, then $\mathcal{C}_\mathrm{pt}\simeq\mathrm{sVec}$ or $\mathcal{C}\simeq\mathcal{C}_\mathrm{pt}\boxtimes\mathcal{T}$ for some transitive modular tensor category $\mathcal{T}$.

\begin{lemma}\label{ad}
Let $\mathcal{C}$ be a modular tensor category.  If $|\mathrm{Orb}(\mathcal{C})|=2$ and $\mathcal{C}_\mathrm{pt}$ is nontrivial, then $\mathcal{O}_\mathbbm{1}=\mathcal{O}(\mathcal{C}_\mathrm{ad})$.
\end{lemma}

\begin{proof}
The relative centralizer $C_\mathcal{C}(\mathcal{C}_\mathrm{pt})=\mathcal{C}_\mathrm{ad}$ must be closed under Galois conjugacy by Theorem \ref{prelem0}.  But $C_\mathcal{C}(\mathcal{C}_\mathrm{pt})\subsetneq\mathcal{C}$ since $\mathcal{C}$ is modular, thus $\mathcal{O}(C_\mathcal{C}(\mathcal{C}_\mathrm{pt}))=\mathcal{O}(\mathcal{C}_\mathrm{ad})=\mathcal{O}_\mathbbm{1}$.
\end{proof}

\begin{proposition}\label{firstwan}
Let $\mathcal{C}$ be a modular tensor category with $|\mathrm{Orb}(\mathcal{C})|=2$.  If $\mathcal{C}_\mathrm{pt}\neq\mathrm{Vec}$, then $\mathrm{FPdim}(\mathcal{C}_\mathrm{pt})=2$ and $\mathcal{C}_\mathrm{pt}$ is symmetric, or $\mathcal{C}\simeq\mathcal{C}_\mathrm{pt}\boxtimes\mathcal{T}$ where $\mathcal{T}$ is a transitive modular tensor category.
\end{proposition}

\begin{proof}
If $\mathcal\mathcal{O}(\mathcal{C}_\mathrm{pt})\not\subset\mathcal{O}_\mathbbm{1}$, then $\mathcal{C}\simeq\mathcal{C}_\mathrm{pt}\boxtimes\mathcal{T}$ where $\mathcal{T}$ is a transitive modular tensor category by Theorem \ref{what2}.  Otherwise $\mathcal\mathcal{O}(\mathcal{C}_\mathrm{pt})\subset\mathcal{O}_\mathbbm{1}$ and thus $ \mathcal{C}_\mathrm{pt}$ is self-dual \cite[Lemma 3.2]{MR3632091} and symmetrically braided by Lemma \ref{counting}.  Lastly Theorem \ref{what} implies $\mathrm{FPdim}(\mathcal{C})$ is prime.  But a pointed self-dual fusion category must be an elementary abelian $2$-group, hence $\mathrm{FPdim}(\mathcal{C}_\mathrm{pt})=2$.
\end{proof}

To finish this subsection by eliminating the possibility that $|\mathrm{Gal}(\mathcal{C})|=2$ and $\mathcal{C}_\mathrm{pt}$ is Tannakian, we will need the standard constructions of connected \'etale algebras in braided fusion categories, and their categories of local modules.  We refer the reader to \cite[Section 3]{DMNO} where comprehensive detail can be found.

\begin{lemma}\label{venti}
Let $\mathcal{C}$ be a braided tensor category and $A$ be the regular algebra of a Tannakian fusion subcategory.  If $\dim(X)$ is an algebraic unit, then the free $A$-module $A\otimes X$ is simple.
\end{lemma}

\begin{proof}
Let $V$ be a simple $A$-submodule of the free $A$-module $A\otimes X$.  The only possible subobjects of $A\otimes X$ are of the form $V=\delta\otimes X$ where $\delta$ is a sum of invertible simple objects, hence $\dim(V)=m\dim(X)$ where $1\leq m\leq\dim(A)$.  We compute $\dim_A(V)$, the dimension of $V\in\mathcal{O}(\mathcal{C}_A)$, using \cite[Theorem 1.18]{KiO}:
\begin{equation}\label{veinte}
\dim_A(V)=\dim(V)/\dim(A)=(m/\dim(A))\dim(X).
\end{equation}
But $\dim(X)$ is an algebraic unit by assumption, thus $m=\dim(A)$ and moreover $A\otimes X$ is simple.
\end{proof}

\begin{lemma}\label{medi}
Let $\mathcal{C}$ be a modular tensor category.  If $\mathcal{O}_\mathbbm{1}=\mathcal{O}(\mathcal{C}_\mathrm{ad})$ and $\mathcal{C}_\mathrm{pt}$ is Tannakian, then $\mathcal{C}_A^0$ is a transitive modular tensor category where $A\in\mathcal{C}_\mathrm{pt}$ is the regular algebra of $\mathcal{C}_\mathrm{pt}$.
\end{lemma}

\begin{proof}
Lemma \ref{venti} implies that every $V\in\mathcal{O}(\mathcal{C}_A^0)$ is free.  Hence every $V\in\mathcal{O}(\mathcal{C}_A^0)$ has squared dimension of the form $\dim(\mathcal{C})/\sigma(\dim(\mathcal{C}))=\dim(\mathcal{C}_A^0)/\sigma(\dim(\mathcal{C}_A^0))$ for some $\sigma\in\mathrm{Gal}(\overline{\mathbb{Q}}/\mathbb{Q})$.  Therefore each Galois orbit of $\mathcal{C}_A^0$ contains a pseudoinvertible object, and Theorem \ref{what2} implies $\mathcal{C}_A^0\simeq(\mathcal{C}_A^0)_\mathrm{pt}\boxtimes\mathcal{T}$ for some transitive modular tensor category $\mathcal{T}$.  Hence $\mathcal{C}_A^0$ is Galois conjugate to a pseudounitary modular tensor category and thus $\mathcal{C}$ is Galois conjugate to a pseudounitary modular tensor category by \cite[Lemma 5.3(c)]{MR3997136} and \cite[Lemma 3.11]{DMNO}.  Lastly, we may now assume $X\in\mathcal{O}(\mathcal{C})$ has $\dim(X)^2=1$ if and only if $X\in\mathcal{O}(\mathcal{C}_\mathrm{pt})$.  Therefore $V\in\mathcal{O}(\mathcal{C}_A^0)$ has $\dim(V)^2=1$ if and only if $V=A$.  So $\mathcal{C}_A^0$ has no nontrivial invertible objects, and is moreover transitive.
\end{proof}

\begin{proposition}\label{equi}
Let $\mathcal{C}$ be a modular tensor category.  If $\mathcal{O}_\mathbbm{1}=\mathcal{O}(\mathcal{C}_\mathrm{ad})$ and $\mathcal{C}_\mathrm{pt}$ is Tannakian, then $\mathcal{C}$ is transitive.  
\end{proposition}

\begin{proof}
Assume $\mathcal{C}_\mathrm{pt}$ is Tannakian with regular algebra $A$.  Then Lemma \ref{medi} implies $\mathcal{C}_A^0$ is transitive.  By \cite[Theorem II]{2020arXiv200701366N}, each transitive modular tensor category has a unique (up to ordering) nontrivial factorization
\begin{equation}
\mathcal{T}\simeq\boxtimes_{j\in J}\mathcal{C}(\mathfrak{sl}_2,p_j-2)_\mathrm{ad}
\end{equation}
where $J$ is some finite index set of distinct primes greater than 3.  Note that de-equivariantization commutes with taking centralizers \cite[Proposition 4.30]{DGNO}.  In our case this means $\mathcal{T}\simeq(\mathcal{C}_A)_\mathrm{ad}\simeq(\mathcal{C}_\mathrm{ad})_A$, so there exists a canonical equivalence $\mathcal{T}^G\simeq\mathcal{C}_\mathrm{ad}$. This equivariantization comes from a braided action of $G:=\mathcal{O}(\mathcal{C}_\mathrm{pt})$ on $\mathcal{T}$ \cite[Theorem 8.23.3]{tcat}; let $\rho:G\to\mathrm{Aut}_\otimes^\mathrm{br}(\mathcal{T})$ be the corresponding group homomorphism.  Theorem 5.2 of \cite{MR2677836} states that $\mathrm{Aut}_\otimes^\mathrm{br}(\mathcal{T})\cong\mathrm{Pic}(\mathcal{T})$ where $\mathrm{Pic}(\mathcal{T})$ is the Picard group of equivalence classes of invertible $\mathcal{T}$-module categories.  The categories $\mathcal{C}(\mathfrak{sl}_2,p_j-2)_\mathrm{ad}$ are pairwise coprime for $j\in J$, so by \cite[Proposition 4.10]{MR2677836},
\begin{equation}
\mathrm{Pic}(\mathcal{T})\cong\times_{j\in J}\mathrm{Pic}(\mathcal{C}(\mathfrak{sl}_2,p_j-2)_\mathrm{ad}).
\end{equation}
But these Picard groups are trivial for all $j\in J$ by \cite[Theorem 1.2]{MR3808052}, so the homomorphism $\rho$ is trivial.  Moreover $\mathcal{C}_\mathrm{ad}\simeq\mathcal{T}^G\simeq\mathcal{T}\boxtimes\mathrm{Rep}(G)$ as braided fusion categories.  Therefore $\mathcal{C}\simeq\mathcal{T}\boxtimes C_\mathcal{C}(\mathcal{T})$ with $C_\mathcal{C}(\mathcal{T})$ modular \cite[Proposition 4.1]{mug1} and weakly integral, and $\mathrm{FPdim}(C_\mathcal{C}(\mathcal{T}))=\mathrm{FPdim}(\mathcal{C}_\mathrm{pt})^2$.

\par Lastly we will prove $\mathcal{C}_\mathrm{pt}$ is trivial, hence $\mathcal{C}$ is transitive.  To this end, note that $\mathcal{O}(\mathcal{C}_\mathrm{pt})\subset\mathcal{O}(\mathcal{C}_\mathrm{ad})=\mathcal{O}_\mathbbm{1}$, which is self-dual.  Hence $\mathcal{O}(\mathcal{C}_\mathrm{pt})$ is an elementary abelian 2-group.  This implies $\mathrm{FPdim}(C_\mathcal{C}(\mathcal{T}))=2^n$ for some $n\in\mathbb{Z}_{\geq0}$ and therefore $\dim(X)\in\mathbb{Q}(\sqrt{2})$ for all $X\in C_\mathcal{C}(\mathcal{T})$.  By Lemma \ref{lem1}, $|\mathcal{O}_\mathbbm{1}|$ is either $\mathrm{rank}(\mathcal{T})$ (when $C_\mathcal{C}(\mathcal{T})$ is integral) or $2\cdot\mathrm{rank}(\mathcal{T})$ (when there exists $X\in\mathcal{O}(C_\mathcal{C}(\mathcal{T}))$ with $\mathbb{Q}(\dim(X))=\mathbb{Q}(\sqrt{2}))$.  In the former case, $C_\mathcal{C}(\mathcal{T})_\mathrm{ad}$ is trivial, so $C_\mathcal{C}(\mathcal{T})$ is pointed.  But this implies $\mathcal{C}_\mathrm{pt}$ is not Tannakian unless $\mathcal{C}_\mathrm{pt}$ is trivial.  The latter case implies $C_\mathcal{C}(\mathcal{T})_\mathrm{ad}$ is rank 2, which then must be pointed.  Moreover, the Ising categories are the only modular tensor categories with the above properties, none of which contain a non-trivial Tannakian subcategory, so this cannot occur.
\end{proof}


\subsection{The case of $\mathrm{sVec}$}\label{svec}

\begin{lemma}\label{porque}
Let $\mathcal{C}$ be a modular tensor category with $|\mathrm{Orb}(\mathcal{C})|=2$ and conductor $N\in\mathbb{Z}_{\geq1}$.  If $16\mid N$, then $N=16m$ for an odd square-free $m\in\mathbb{Z}_{\geq3}$.  Moreover for some multiplicity $n\in\mathbb{Z}_{\geq1}$, $\rho_\mathcal{C}\cong n(\psi_e\otimes\psi_o)$, where $\psi_o$ is an irreducible transitive representation of $\mathrm{SL}(2,\mathbbm{Z}/m\mathbb{Z})$, and $\psi_e$ is a 3-dimensional irreducible representation of $\mathrm{SL}(2,\mathbbm{Z}/16\mathbb{Z})$.
\end{lemma}

\begin{proof}
Let $\psi$ be an irreducible summand of $\rho_\mathcal{C}$ whose level is divisible by $16$.  Then $\psi\cong\psi_e\otimes\psi_o$ where $\psi_e$ is an irreducible representation of level $2^n$ for some $n\in\mathbb{Z}_{\geq3}$ and $\psi_o$ is an irreducible representation of level $m$ for some odd integer $m$ by the Chinese remainder theorem.  Since the number of square Galois orbits of $\mathfrak{t}$-eigenvalues is multiplicative across coprime orders, then exactly one each of $\psi_e,\psi_o$ have 1 or 2 square Galois orbits of $\mathfrak{t}$-eigenvalues.  But there do not exist transitive irreducible representations of level $2^\lambda$ for $\lambda\geq4$ (Appendix \ref{tables}).  This implies that $\psi_o$ is transitive, i.e. $m$ is square-free.  Moreover, $\psi_e$ has two square Galois orbits of $\mathfrak{t}$-eigenvalues, so $\psi_e$ is isomorphic to one of the sixteen 3-dimensional irreducible representations of $\mathrm{SL}(2,\mathbbm{Z}/16\mathbb{Z})$  (Appendix \ref{tables}).

\par Let $\varphi$ be any other irreducible summand of $\rho_\mathcal{C}$.  The eigenvalues of $\varphi(\mathfrak{t})$ must be a subset of those of $\psi$, or else $|\mathrm{Orb}(\mathcal{C})|>2$.  In the case $\mathrm{ord}(\mathfrak{t}_\mathcal{C})=16m$, the eigenvalues of $\psi(\mathfrak{t})$ are all primitive $16m$-th roots of unity, hence $\varphi$ is an irreducible representation of level $16m$ as well.  By the above argument, $\varphi\cong\psi$ since they have the same eigenvalues.
\end{proof}

\begin{lemma}\label{linearalgebra}
Let $\rho$ be an irreducible representation of $\mathrm{SL}(2,\mathbb{Z}/N\mathbb{Z})$ for some $N\in\mathbb{Z}_{\geq2}$ and $\mathcal{C}$ be a modular tensor category.  If $\rho$ has distinct eigenvalues, then $\rho_\mathcal{C}\not\cong n\rho$ for any $n\in\mathbb{Z}_{\geq2}$.
\end{lemma}

\begin{proof}
Assume to the contrary that $\rho_\mathcal{C}\cong n\rho$.  Set $r:=\mathrm{rank}(\mathcal{C})=n\dim(\rho)$.  Let $Y$ be the block diagonal matrix of $n$ copies of $\rho(\mathfrak{s})$, $Z$ be the block diagonal matrix of $n$ copies of $\rho(\mathfrak{t})$, and $P$ be the $r\times r$ change-of-basis matrix such that $Y=P^{-1}\rho_\mathcal{C}(\mathfrak{s})P$ and $Z=P^{-1}\rho_\mathcal{C}(\mathfrak{t})P$.  By possibly commuting with a permutation matrix, we may assume that $P$ is such that $\rho_\mathcal{C}(\mathfrak{t})=P^{-1}\rho_\mathcal{C}(\mathfrak{t})P$ on the nose, i.e. $P\rho_\mathcal{C}(\mathfrak{t})=\rho_\mathcal{C}(\mathfrak{t})P$.    Consider $P$ and $P^{-1}$ as matrices of $\dim(\rho)\times\dim(\rho)$ blocks $P_{jk}$ and $R_{jk}$ for $1\leq j,k\leq n$, respectively.  Then for all $1\leq j,k\leq n$, $P\rho_\mathcal{C}(\mathfrak{t})=\rho_\mathcal{C}(\mathfrak{t})P$ implies $P_{jk}\rho(\mathfrak{t})=\rho(\mathfrak{t})P_{jk}$, and $\rho_\mathcal{C}(\mathfrak{t})P^{-1}=P^{-1}\rho_\mathcal{C}(\mathfrak{t})$ implies $\rho(\mathfrak{t})R_{jk}=R_{jk}\rho(\mathfrak{t})$.  Since $\rho(\mathfrak{t})$ has distinct eigenvalues, these commutations imply $P_{jk}$ and $R_{jk}$ are diagonal matrices for all $1\leq j,k\leq n$.  We compute a generic block of $PY$ as
\begin{equation}
(PY)_{jk}=\sum_{\ell=1}^nP_{j\ell}(Y)_{\ell k}=P_{jk}\rho(\mathfrak{s}),
\end{equation}
hence for $1\leq j,k\leq n$,
\begin{equation}
(\rho_\mathcal{C}(\mathfrak{s}))_{jk}=(PYP^{-1})_{jk}=\sum_{\ell=1}^n P_{j\ell}\rho(\mathfrak{s})R_{\ell k}.
\end{equation}
But $PP^{-1}=I_{r\times r}$ implies $\sum_{\ell=1}^nP_{j\ell}R_{\ell k}=0$ for $j\neq k$.  Therefore, the diagonal elements of $(\rho_\mathcal{C}(\mathfrak{s}))_{jk}$ are 0 for all $j\neq k$, and since $\rho_\mathcal{C}(\mathfrak{s})$ is symmetric, each row/column of $\rho_\mathcal{C}(\mathfrak{s})$ contains at least one 0 entry.  This cannot occur since the row/column corresponding to the Frobenius-Perron dimension must consist only of positive real numbers.
\end{proof}

\begin{lemma}\label{preleminf}
Let $\mathcal{C}$ be a modular tensor category with $|\mathrm{Orb}(\mathcal{C})|=2$.  If $\mathcal{C}_\mathrm{pt}\simeq\mathrm{sVec}$, then $\rho_\mathcal{C}\cong\psi_e\otimes\psi_o$, where $\psi_o$ is an irreducible transitive representation of $\mathrm{SL}(2,\mathbbm{Z}/m\mathbb{Z})$ as in Lemma \ref{porque}, and $\psi_e$ is a 3-dimensional irreducible representation of $\mathrm{SL}(2,\mathbbm{Z}/16\mathbb{Z})$.
\end{lemma}

\begin{proof}
We begin by noting that $\mathcal{C}_\mathrm{pt}\subset\mathcal{O}_\mathbbm{1}$, or else $\mathcal{C}_\mathrm{pt}$ is nondegenerately braided by Theorem \ref{what2}, a contradiction.  Let $\sigma\in\mathrm{Gal}(\overline{\mathbb{Q}}/\mathbb{Q})$ such that $\hat{\sigma}(\mathbbm{1})$ is the nontrivial invertible object of $\mathrm{sVec}$.  In particular, $1=\theta_\mathbbm{1}=-\theta_{\hat{\sigma}(\mathbbm{1})}=-1$, hence $-t_\mathbbm{1}=t_{\hat{\sigma}(\mathbbm{1})}=\sigma^2(t_\mathbbm{1})$.  The only roots of unity which are square Galois conjugate to their negative have order divisible by 16, so our claim follows from Lemma \ref{porque} and Lemma \ref{linearalgebra}.
\end{proof}

\begin{proposition}\label{eqqq}
Let $\mathcal{C}$ be a modular tensor category with $|\mathrm{Orb}(\mathcal{C})|=2$.  If $\mathcal{C}_\mathrm{pt}\simeq\mathrm{sVec}$, then $\mathcal{C}\simeq\mathcal{I}\boxtimes\mathcal{T}$ where $\mathcal{I}$ is an Ising modular tensor category and $\mathcal{T}$ is a transitive modular tensor category.
\end{proposition}

\begin{proof}
Lemma \ref{preleminf} states that $\rho_\mathcal{C}\cong\psi_e\otimes\psi_o$ where $\psi_o$ is an irreducible transitive representation of $\mathrm{SL}(2,\mathbbm{Z}/m\mathbb{Z})$ for some square-free $m\in\mathbb{Z}_{\geq3}$, and $\psi_e$ is a 3-dimensional irreducible representation of $\mathrm{SL}(2,\mathbbm{Z}/16\mathbb{Z})$.  In particular, the eigenvalues of $\rho_\mathcal{C}(\mathfrak{t})$ are distinct and so $\mathcal{C}$ is self-dual.  Up to a change of basis, we have
\begin{equation}
\psi_e(\mathfrak{s})=\dfrac{\zeta_4^k}{2}\left[\begin{array}{ccc}0 & \sqrt{2} & -\sqrt{2}\\ \sqrt{2} &1 & 1 \\
-\sqrt{2} &1 & 1\end{array}\right]\label{eqq}
\end{equation}
where $\zeta^k_4=\exp(2k\pi i/4)$ for some $k=0,1,2,3$ \cite[Table A2]{MR1354262}.  Recall that the irreducible representation $\psi_o$ is realized by a modular tensor category $\mathcal{D}$ which is equivalent to the Deligne product of transitive modular tensor categories $\mathcal{C}(\mathfrak{sl}_2,p-2)^{\sigma_p}_\mathrm{ad}$ for some Galois automorphism $\sigma_p\in\mathrm{Gal}(\mathbb{Q}(\zeta_p)/\mathbb{Q})$ for all $p\mid m$.  So we may assume that $\psi_o(\mathfrak{s})$ is the normalized $S$-matrix of $\mathcal{D}$ by a change of basis independent to that for $\psi_e(\mathfrak{s})$ since their $\mathfrak{t}$-eigenvalues are disjoint.  In particular, $\psi_o(\mathfrak{s})$ is real, so $k\in\{0,2\}$ in (\ref{eqq}).   Since $\mathcal{D}$ is transitive, every row/column of $\psi_o(\mathfrak{s})$ contains a unique entry of the form $\pm1/\sqrt{\dim(\mathcal{D})}$.  By observing the rows/columns without zeroes of $\psi_e(\mathfrak{s})$, we see that $1/\sqrt{\dim(\mathcal{C})}$ is either $1/\sqrt{2\dim(\mathcal{D})}$ or $1/\sqrt{4\dim(\mathcal{D})}$, i.e.\ $\dim(\mathcal{C})=2\dim(\mathcal{D})$ or $\dim(\mathcal{C})=4\dim(\mathcal{D})$.  But since every row/column of $\rho_\mathcal{C}(\mathfrak{s})$ contains a unique entry of the form $\pm1/\sqrt{2\dim(\mathcal{D})}$, then the former would imply there exists a unique pseudoinvertible object $X\in\mathcal{O}(\mathcal{C})$.  The pointed subcategory $\mathcal{C}_\mathrm{pt}$ is assumed nontrivial, so we may conclude $\dim(\mathcal{C})=4\dim(\mathcal{D})$ and moreover, $\mathcal{C}$ possesses a unique, simple, self-dual object $X$ with $\dim(X)=\sqrt{2}$.  It is clear that $X$ must $\otimes$-generate a nondegenerately braided fusion subcategory of $\mathcal{C}$ which is equivalent to one of the Ising modular tensor categories and therefore $\mathcal{C}$ factorizes as claimed.
\end{proof}

\begin{corollary}\label{pointedprop}
Let $\mathcal{C}$ be a modular tensor category.  Then $|\mathrm{Orb}(\mathcal{C})|=2$ and $\mathcal{C}_\mathrm{pt}$ is nontrivial if and only if $\mathcal{C}\simeq\mathcal{D}\boxtimes\mathcal{T}$ is an equivalence of modular tensor categories, where $\mathcal{T}$ is a transitive modular tensor category, and $\mathcal{D}$ is either pointed of prime dimension coprime to the conductor of $\mathcal{T}$, or an Ising modular tensor category.
\end{corollary}

\begin{proof}
If $|\mathrm{Orb}(\mathcal{C})|=2$ and $\mathcal{C}_\mathrm{pt}$ is nontrivial, then Propositions \ref{firstwan} and \ref{equi} imply that either $\mathcal{C}_\mathrm{pt}\simeq\mathrm{sVec}$, or $\mathcal{C}\simeq\mathcal{D}\boxtimes\mathcal{T}$ is an equivalence of modular tensor categories, where $\mathcal{T}$ is a transitive modular tensor category, and $\mathcal{D}$ is pointed of prime dimension coprime to the conductor of $\mathcal{T}$.  In the former case, Proposition \ref{eqqq} implies $\mathcal{C}\simeq\mathcal{I}\boxtimes\mathcal{T}$ where $\mathcal{I}$ is an Ising modular tensor category and $\mathcal{T}$ is a transitive modular tensor category.  The converse direction is trivial.
\end{proof}


\subsection{The unpointed case}

\begin{lemma}\label{nondeg}
Let $\mathcal{C}$ be an unpointed modular tensor category with $|\mathrm{Orb}(\mathcal{C})|=2$.  If $\mathcal{D}\subset\mathcal{C}$ is a fusion subcategory, then $\mathcal{D}$ is a modular tensor category.
\end{lemma}

\begin{proof}
Let $\mathcal{D}\subset\mathcal{C}$ be a fusion subcategory.  Then $\mathcal{D}$ inherits the braiding of $\mathcal{C}$ and we may consider the symmetric center $\mathcal{D}'$.  We know that $\mathcal{D}'$ is integral, hence pseudounitary, thus $\mathrm{FPdim}(X)=\pm\dim(X)$ for all $X\in\mathcal{O}(\mathcal{D}')$.  If $X,Y\in\mathcal{O}(\mathcal{D}')$ are Galois conjugate (by the Galois action on $\mathcal{C}$), then there exists $\sigma\in\mathrm{Gal}(\overline{\mathbb{Q}}/\mathbb{Q})$ such that
\begin{equation}
\dim(X)^2=\sigma(\dim(Y)^2)\dim(\mathcal{C})/\sigma(\dim(\mathcal{C}))=\dim(Y)^2\dim(\mathcal{C})/\sigma(\dim(\mathcal{C})).
\end{equation}
Thus $\dim(X)^2/\dim(Y)^2=\dim(\mathcal{C})/\sigma(\dim(\mathcal{C}))=1$ as it is a positive integer unit \cite[Corollary 1.4]{MR2576705}.  Hence $\dim(X)=\pm\dim(Y)$ and moreover $\mathrm{FPdim}(X)=\mathrm{FPdim}(Y)$.  Thus there are at most two distinct Frobenius-Perron dimensions of simple objects in $\mathcal{D}'$.
\par Let $X\in\mathcal{O}(\mathcal{D}')\setminus\{\mathbbm{1}\}$.  Then since $\mathbbm{1}$ is the unique simple object of Frobenius-Perron dimension 1, by the above argument,
\begin{equation}
\mathrm{FPdim}(X)^2=\mathrm{FPdim}(X\otimes X^\ast)=1+n\mathrm{FPdim}(X)
\end{equation}
for some $n\in\mathbb{Z}_{\geq1}$.  This would imply $1\equiv0\pmod{\mathrm{FPdim}(X)}$, therefore no such $X$ exists.  Moreover $\mathcal{D}'\simeq\mathrm{Vec}$.
\end{proof}

\begin{proposition}\label{unpointedprop}
Let $\mathcal{C}$ be a modular tensor category.  Then $|\mathrm{Orb}(\mathcal{C})|=2$ and $\mathcal{C}_\mathrm{pt}$ is trivial if and only if $\mathcal{C}\simeq\mathcal{D}\boxtimes\mathcal{T}$ is an equivalence of modular tensor categories, where $\mathcal{T}$ is a transitive modular tensor category, and $\mathcal{D}$ is either a simple modular tensor category whose conductor is coprime to that of $\mathcal{T}$, or $\mathrm{rank}(\mathcal{T})>2$ and $\mathcal{D}$ is braided equivalent to $\mathcal{F}_1\boxtimes\mathcal{F}_2$ where $\mathcal{F}_1,\mathcal{F}_2$ are any of the Fibonacci modular tensor categories.
\end{proposition}

\begin{proof}
For the forward direction, Lemma \ref{nondeg} implies the unique factorization (up to braided equivalence) given by \cite[Proposition 2.2]{DMNO} consists of simple factors.  Lemma 2.1(iii) of \cite{2020arXiv200701366N} then implies at most one of these factors, $\mathcal{E}$, has 2 Galois orbits of simple objects.  If $\mathcal{E}$ exists, then the remaining factors are transitive and coprime, or else $|\mathrm{Orb}(C_\mathcal{C}(\mathcal{E}))|>1$ by Example \ref{transdeligne} which would imply $|\mathrm{Orb}(\mathcal{C})|\geq4$ by \cite[Lemma 2.1(iii)]{2020arXiv200701366N}.  In this case $\mathcal{E}=\mathcal{D}$ proves our claim.  If $\mathcal{E}$ does not exist, then all simple factors of $\mathcal{C}$ are transitive, and at least 2 are \emph{not} coprime or else $\mathcal{C}$ would be transitive.  Example \ref{transdeligne} implies that $\mathcal{D}\simeq\mathcal{F}_1\boxtimes\mathcal{F}_2$ for some Fibonacci modular tensor categories $\mathcal{F}_1,\mathcal{F}_2$, is the only product of non-coprime simple transitive modular tensor categories with 2 Galois orbits of simple objects.  The converse direction is trivial.
\end{proof}


\begin{appendices}
\section{$\mathfrak{t}$-spectra}\label{tables}

Here we include the classification of irreducible $\mathrm{SL}(2,\mathbb{Z}/N\mathbb{Z})$ representations for $N\in\mathbb{Z}_{\geq2}$.  This classification first appeared in \cite[Section 9]{MR444788} but to justify the proofs of Section \ref{svec}, we have contributed the $\mathfrak{t}$-eigenvalues of each representation as a set as well as the number of square Galois orbits of roots of unity in this set.  It is not necessary for our purposes to include multiplicities of eigenvalues and doing so would increase the length of our exposition greatly.

\par In the following tables we denote the set of all $n$-th roots of unity by $\Phi_n$, and the set of all primitive $n$-th roots of unity by $\Gamma_n$, for brevity.  We denote the square Galois orbit of $\zeta_{p^\lambda}^r$ by $\Gamma_{p^\lambda}^r$ for any $r\in\mathbb{Z}$.  Any squared Galois orbit of primitive $p^\lambda$th roots of unity for prime $p$ and $\lambda\in\mathbb{Z}_{\geq1}$ is the orbit of $\zeta_{p^\lambda}^r$ where $r$ is a quadratic (non)-residue modulo $p$ when $p$ is odd, or the orbit of $\zeta_{2^\lambda}^r$ where $r\in\{1,3,5,7\}$.  In particular, $\Phi_1=\Gamma_1=\{1\}$, $\Gamma_2=\{-1\}$, $\Gamma_{2^2}=\Gamma_{2^2}^1\cup\Gamma_{2^2}^3=\{\pm i\}$, $\Gamma_{2^\lambda}=\Gamma_{2^\lambda}^1\cup\Gamma_{2^\lambda}^3\cup\Gamma_{2^\lambda}^5\cup\Gamma_{2^\lambda}^7$ when $\lambda\geq3$, and $\Gamma_{p^\lambda}=\Gamma_{p^\lambda}^r\cup\Gamma_{p^\lambda}^s$ for odd $p$, where $r$ is a quadratic residue and $s$ is a quadratic non-residue modulo $p$.  The sets $\Gamma^{r_1}_{{p_1}^{\lambda_1}}$ and $\Gamma^{r_2}_{{p_2}^{\lambda_2}}$ are disjoint unless $r_1=r_2$, $p_1=p_2$ and $\lambda_1=\lambda_2$ (similarly for $\Gamma_{{p_1}^{\lambda_1}}$ and $\Gamma_{{p_2}^{\lambda_2}}$), hence all the unions in the following tables are disjoint unions.
\par Denote the one-dimensional modular group representations whose set of $\mathfrak{t}$-eigenvalues are $\Gamma_1=\{1\}$, $\Gamma_2=\{-1\}$, $\Gamma^3_{2^2}=\{-i\}$, $\Gamma^1_{2^4}=\{i\}$ by $C_1,C_2,C_3,C_4$, respectively.  If a representation has distinct $\mathfrak{t}$-eigenvalues then we say it is ``multiplicity-free'', denoted m.f. in the tables.  We have abbreviated the number of square Galois orbits in the set of $\mathfrak{t}$-eigenvalues by $|\mathrm{Gal}|$; only a lower bound is given in exactly two cases to save space in an already lengthy exposition.  Unless otherwise noted in column 2, $\chi$ will represent an arbitrary character in $\mathfrak{B}$, the set of characters of the automorphism group of the binary quadratic module associated to the representation \cite[Section 2]{MR444788}.

\[
\begin{array}{llllll}
\multicolumn{6}{l}{\textbf{Table 1. }\text{Irreducible representations of }\mathrm{SL}(2,\mathbb{Z}/p\mathbb{Z})\text{ for }p\neq2} \\
\hline \text{type of rep.} & & \text{dim} & \mathfrak{t}-\text{spectrum}& \text{m.f.}& |\text{Gal}|\\\hline 
D_1(\chi) &  & p+1 & \Phi_p & \mathrm{no} &3 \\
N_1(\chi) & & p-1 & \Gamma_p & \mathrm{yes} &2 \\
R_1(r,\chi_1) & \left(\dfrac{r}{p}\right)=\pm1 & \frac{1}{2}(p+1) & \Gamma_1\cup\Gamma_p^r  & \mathrm{yes} &2 \\
R_1(r,\chi_{-1}) & \left(\dfrac{r}{p}\right)=\pm1 & \frac{1}{2}(p-1) & \Gamma_p^r & \mathrm{yes} &1 \\
N_1(\chi_1) &  & p & \Phi_p & \mathrm{yes} &3 \\
\hline
\end{array}
\]

\[
\begin{array}{llllll}
\multicolumn{6}{l}{\textbf{Table 2. }\text{Irreducible representations of }\mathrm{SL}(2,\mathbb{Z}/p^\lambda\mathbb{Z})\text{ for }p\neq2,\lambda>1, 1\leq\sigma<\lambda,\text{ and }\left(\dfrac{r}{p}\right),\left(\dfrac{t}{p}\right)=\pm1} \\
\hline \text{type of rep.} & & \text{dim} & \mathfrak{t}-\text{spectrum}& \text{m.f.} & |\text{Gal}| \\\hline 
D_\lambda(\chi) &  & (p+1)p^{\lambda-1} & \Phi_{p^\lambda} & \mathrm{no} & 2\lambda+1 \\
N_\lambda(\chi) &  & (p-1)p^{\lambda-1} &\Gamma_{p^\lambda} & \mathrm{yes} & 2 \\
R_\lambda^\sigma(r,t,\chi) &  & \frac{1}{2}(p^2-1)p^{\lambda-2}& \Gamma^r_{p^\lambda}\cup\left\{\zeta_{p^\lambda}^{r(x^2+p^\sigma ty^2)}:x,y\in\mathbb{Z},p\nmid y,p\mid x\right\}  &  \left\{\begin{array}{ccc}\text{yes} & : & \sigma=1 \\ \text{no} & : &\mathrm{else}\end{array}\right. & \geq\sigma+1\\
R_\lambda(r,\chi_\pm)_1 & &\frac{1}{2}(p^2-1)p^{\lambda-2} & \Gamma_{p^\lambda}^r\cup\Phi^r_{p^{\lambda-2}} & \left\{\begin{array}{ccc}\text{yes} & : & \lambda=2, \\ & & p=3 \\ \text{no} & : & \text{else}\end{array}\right. & \lambda \\
\hline
\end{array}
\]

\newpage

\[
\begin{array}{llllll}
\multicolumn{6}{l}{\textbf{Table 3. }\text{Irreducible representations of }\mathrm{SL}_2(2,\mathbb{Z}/2\mathbb{Z})} \\
\hline \text{type of rep.} & & \text{dim} & \mathfrak{t}-\text{spectrum}& \text{m.f.} & |\text{Gal}| \\\hline 
C_2=N_1(\chi) &  & 1 & \Gamma_2 & \text{yes} &1 \\
N_1(\chi_1) &  & 2 & \Phi_2 & \text{yes}  & 2\\
\hline
\end{array}
\]

\[
\begin{array}{llllll}
\multicolumn{6}{l}{\textbf{Table 4. }\text{Irreducible representations of }\mathrm{SL}_2(2,\mathbb{Z}/4\mathbb{Z})} \\
\hline \text{type of rep.} & & \text{dim} & \mathfrak{t}-\text{spectrum}& \text{m.f.} &|\text{Gal}|\\\hline 
R_2^0(1,1,\chi_1) &  & 3 & \Phi_2\cup\Gamma^1_{2^2} & \text{yes} & 3 \\
R_2^0(3,1,\chi_1) &  & 3 & \Phi_2\cup\Gamma^3_{2^2} & \text{yes} & 3 \\
R_2^0(1,3)_1 &  & 3 & \Gamma_1\cup\Gamma_{2^2} & \text{yes} & 3 \\
C_2\otimes R_2^0(1,3)_1 &  & 3 & \Gamma_2\cup\Gamma_{2^2} & \text{yes} & 3 \\
N_2(\chi) & \chi\not\equiv1 & 2 & \Gamma_{2^2} & \text{yes} & 2 \\
C_3=R_2^0(3,1,\chi) & \chi\not\equiv1 & 1 & \Gamma^3_{2^2} & \text{yes} & 1 \\
C_4=R_2^0(1,1,\chi) & \chi\not\equiv1 & 1 & \Gamma^1_{2^2} & \text{yes} &1  \\
\hline
\end{array}
\]

\[
\begin{array}{llllll}
\multicolumn{6}{l}{\textbf{Table 5. }\text{Irreducible representations of }\mathrm{SL}_2(2,\mathbb{Z}/8\mathbb{Z})} \\
\hline \text{type of rep.} & & \text{dim} & \mathfrak{t}-\text{spectrum}&\text{m.f.}& |\text{Gal}| \\\hline 
R_3^1(r,t,\chi_1) &  & 6 & \Gamma_{2^2}\cup\left\{\begin{array}{ccc}\Gamma^1_{2^3}\cup\Gamma^3_{2^3} & : & r=1,t=1 \\ \Gamma^1_{2^3}\cup\Gamma^7_{2^3} & : & r=1,t=3\\\Gamma^3_{2^3}\cup\Gamma^5_{2^3} & : & r=3,t=3 \\ \Gamma^5_{2^3}\cup\Gamma^7_{2^3} & : & r=5,t=1\end{array}\right. & \text{yes} & 4 \\
R_3^0(1,3,\chi_1)_1 &  & 6 & \Phi_2\cup\Gamma_{2^3}& \text{yes} & 6\\
C_3\otimes R_3^0(1,3,\chi_1)_1 & & 6 & \Gamma_{2^2}\cup\Gamma_{2^3}& \text{yes}  & 6\\
N_3(\chi) & \chi^2\not\equiv1 & 4 & \Gamma^1_{2^3}\cup\Gamma^3_{2^3}\cup\Gamma^5_{2^3}\cup\Gamma^7_{2^3} & \text{yes}&4\\
C_j\otimes N_3(\chi)_+ & j=1,2,3,4 & 2 & \left\{\begin{array}{ccc} \Gamma_{2^3}^3\cup\Gamma_{2^3}^5& : & j=1 \\ \Gamma_{2^3}^1\cup\Gamma_{2^3}^7 & : & j=2 \\\Gamma_{2^3}^1\cup\Gamma_{2^3}^3& : & j=3 \\\Gamma_{2^3}^5\cup\Gamma_{2^3}^7& : & j=4 \end{array}\right. & \text{yes} & 2\\
C_j\otimes R_3^0(1,3,\chi)_+ & \chi\not\equiv1,j=1,2,3,4 & 3 & \left\{\begin{array}{ccc}\Gamma_2\cup\Gamma^1_{2^3}\cup\Gamma^5_{2^3} & : & j=1\\ \Gamma_1\cup\Gamma^1_{2^3}\cup\Gamma^5_{2^3}& : & j=2\\ \Gamma^1_{2^2}\cup\Gamma^3_{2^3}\cup\Gamma^7_{2^3}& : & j=3\\ \Gamma^3_{2^2}\cup\Gamma^3_{2^3}\cup\Gamma^7_{2^3}& : & j=4\end{array}\right. & \text{yes} & 3 \\
C_j\otimes R_3^0(1,3,\chi)_- & \chi\not\equiv1,j=1,2,3,4 & 3 &\left\{\begin{array}{ccc}\Gamma_2\cup\Gamma^3_{2^3}\cup\Gamma^7_{2^3} & : & j=1\\ \Gamma_1\cup\Gamma^3_{2^3}\cup\Gamma^7_{2^3}& : & j=2\\ \Gamma^1_{2^2}\cup\Gamma^1_{2^3}\cup\Gamma^5_{2^3}& : & j=3\\ \Gamma^3_{2^2}\cup\Gamma^1_{2^3}\cup\Gamma^5_{2^3}& : & j=4\end{array}\right.  & \text{yes} & 3 \\\hline
\end{array}
\]

\begin{equation*}
\begin{array}{llllll}
\multicolumn{6}{l}{\textbf{Table 6. }\text{Irreducible representations of }\mathrm{SL}_2(2,\mathbb{Z}/16\mathbb{Z})}\\
\hline \text{type of rep.} & & \text{dim}& \mathfrak{t}-\text{spectrum}& \text{m.f.} &|\text{Gal}|\\\hline 
D_4(\chi) &  & 24 &\Phi_{2^4} &\text{no} &12 \\
N_4(\chi) &  & 8 &\Gamma_{2^4} &\text{yes} &4 \\
R_4^0(r,t,\chi) & \chi\not\equiv1;r=1,3;t=1,5 & 6 &\Gamma_{2^4}^r\cup\Gamma_{2^4}^{5r}\cup\left\{\begin{array}{ccc}\Gamma^r_{2^3}\cup\Gamma^{5r}_{2^3} &:&t=1 \\ \Gamma^{3r}_{2^3}\cup\Gamma^{7r}_{2^3} &:&t=5\end{array}\right. & \text{yes} & 4\\
C_j\otimes R_4^0(1,1,\chi)_+ & \chi^2\equiv1;j=1,2,3,4 & 3 &\left\{\begin{array}{ccc}\Gamma_{2^3}^5\cup\Gamma_{2^4}^1 & : & j=1 \\ \Gamma_{2^3}^1\cup\Gamma_{2^4}^1 & : & j=2 \\ \Gamma_{2^3}^3\cup\Gamma_{2^4}^5& : & j=3 \\\Gamma_{2^3}^7\cup\Gamma_{2^4}^5 & : & j=4\end{array}\right. &\text{yes} &2 \\
C_j\otimes R_4^0(1,1,\chi)_- & \chi^2\equiv1;j=1,2,3,4 & 3 & \left\{\begin{array}{ccc}\Gamma_{2^3}^1\cup\Gamma_{2^4}^5 & : & j=1 \\ \Gamma_{2^3}^5\cup\Gamma_{2^4}^5 & : & j=2 \\\Gamma_{2^3}^7\cup\Gamma_{2^4}^1 & : & j=3 \\\Gamma_{2^3}^3\cup\Gamma_{2^4}^1 & : & j=4\end{array}\right. &\text{yes} &2 \\
C_j\otimes R_4^0(3,1,\chi)_+ & \chi^2\equiv1;j=1,2,3,4 & 3 & \left\{\begin{array}{ccc}\Gamma_{2^3}^7\cup\Gamma_{2^4}^3 & : & j=1 \\ \Gamma_{2^3}^3\cup\Gamma_{2^4}^3 & : & j=2 \\ \Gamma_{2^3}^5\cup\Gamma_{2^4}^7& : & j=3 \\\Gamma_{2^3}^1\cup\Gamma_{2^4}^7 & : & j=4\end{array}\right. &\text{yes} &2 \\
C_j\otimes R_4^0(3,1,\chi)_- & \chi^2\equiv1;j=1,2,3,4 & 3 &\left\{\begin{array}{ccc}\Gamma_{2^3}^3\cup\Gamma_{2^4}^7 & : & j=1 \\ \Gamma_{2^3}^7\cup\Gamma_{2^4}^7 & : & j=2 \\ \Gamma_{2^3}^1\cup\Gamma_{2^4}^3& : & j=3 \\ \Gamma_{2^3}^5\cup\Gamma_{2^4}^3& : & j=4\end{array}\right. &\text{yes} &2 \\
R_4^0(1,t,\chi)_\pm & t=3,7 & 6 & \Gamma^1_{2^4}\cup\Gamma_{2^4}^7\cup\left\{\begin{array}{ccc} \Gamma_{2^3}^3\cup\Gamma_{2^3}^5 & : & t=3 \\ \Gamma_{2^3}^1\cup\Gamma_{2^3}^7 & : & t=7 \end{array}\right. & \text{yes} &  4 \\
R_4^2(r,t,\chi) & \chi\not\equiv1;r,t\in\{1,3\} & 6 & \left\{\begin{array}{ccc} \Gamma_2\cup\Gamma^1_{2^2}\cup\Gamma_{2^4}^1\cup\Gamma_{2^4}^5 & : & r=1,t=1 \\\Gamma_1\cup\Gamma^3_{2^2}\cup\Gamma_{2^4}^1\cup\Gamma_{2^4}^5 & : & r=1,t=3 \\\Gamma_2\cup\Gamma^3_{2^2}\cup\Gamma_{2^4}^3\cup\Gamma_{2^4}^7 & : & r=3,t=1 \\\Gamma_1\cup\Gamma^1_{2^2}\cup\Gamma_{2^4}^3\cup\Gamma_{2^4}^7 & : & r=3,t=3 \\\end{array}\right. &\text{yes} & 4 \\
C_2\otimes R_4^2(r,3,\chi) & \chi\not\equiv1;r=1,3 & 6 & \left\{\begin{array}{ccc} \Gamma_2\cup\Gamma_{2^2}^1\cup\Gamma_{2^4}^1\cup\Gamma_{2^4}^5 & : & r=1 \\ \Gamma_2\cup\Gamma_{2^2}^3\cup\Gamma_{2^4}^3\cup\Gamma_{2^4}^7 & : & r=3 \end{array}\right. & \text{yes} & 4 \\
R_4^2(r,3,\chi_1)_1 & r=1,3 & 6 & \Gamma_1\cup\Gamma_{2^2}^3\cup\Gamma_{2^4}^1\cup\Gamma_{2^4}^5 &\text{yes} & 4 \\
N_3(\chi)_+\otimes R_4^0(1,7,\psi)_+ & \chi^2\equiv1;\psi\not\equiv1; & 12 & \Gamma_{2}\cup\Gamma_{2^2}\cup\Gamma_{2^4} & \text{no} & 7 \\
 & \psi^2\not\equiv1;\psi(-1)=1 &  &  &  &  \\
\hline
\end{array}\label{table16}
\end{equation*}

\[
\begin{array}{llllll}
\multicolumn{6}{l}{\textbf{Table 7. }\text{Irreducible representations of }\mathrm{SL}_2(2,\mathbb{Z}/32\mathbb{Z})}\\
\hline \text{type of rep.} & & \text{dim} & \mathfrak{t}-\text{spectrum}& \text{m.f.} &|\text{Gal}|\\\hline 
D_5(\chi) &  & 48 & \Phi_{2^5} &\text{no} &16  \\
N_5(\chi) &  & 16 & \Gamma_{2^5} &\text{yes} &4 \\
R_5^0(r,t,\chi) & r=1,3;t=1,5 & 12 & \Gamma_{2^5}^r\cup\Gamma_{2^5}^{5r}\cup\left\{\begin{array}{ccc}\Gamma_{2^4}^r\cup\Gamma_{2^4}^{5r} &:&t=1 \\ \Gamma_{2^4}^{3r}\cup\Gamma_{2^4}^{7r} &:&t=5\end{array}\right. &\text{yes} & 4 \\
R_5^0(1,t,\chi) & t=3,7 & 24 & \Gamma_{2^5}\cup\left\{\begin{array}{ccc}\Gamma_{2^3} &:&t=3 \\ \Phi_{2^2} &:&t=7\end{array}\right. &\text{no} & 8 \\
R_5^1(r,1,\chi) & r=1,5 & 12 & \Gamma^1_{2^5}\cup\Gamma^3_{2^5}\cup\Gamma^r_{2^4}\cup\Gamma^{3r}_{2^4} & \text{yes} &4 \\
R_5^1(r,3,\chi) & r=1,3 & 12 & \Gamma^1_{2^5}\cup\Gamma^7_{2^5}\cup\Gamma^{3r}_{2^4}\cup\Gamma^{5r}_{2^4} & \text{yes} &4 \\
R_5^1(r,5,\chi) & r=1,5 & 12 & \Gamma^1_{2^5}\cup\Gamma^3_{2^5}\cup\Gamma^{5r}_{2^4}\cup\Gamma^{7r}_{2^4} & \text{yes} &4 \\
R_5^1(r,7,\chi) & r=1,3 & 12 & \Gamma^1_{2^5}\cup\Gamma^7_{2^5}\cup\Gamma^r_{2^4}\cup\Gamma^{7r}_{2^4} & \text{yes} &4 \\
R_5^2(r,t,\chi)_\pm & r=1,3;t=1,3,5,7 & 6 & \Gamma_{2^4}\cup\Gamma_{2^5}^r & \text{yes} & 3 \\
R_5^2(r,1,\chi)_1 & \chi\not\in\mathfrak{B};r=1,3 & 12 & \Gamma_1\cup\Gamma_2\cup\Gamma_{2^3}^r\cup\Gamma_{2^3}^{5r}\cup\Gamma_{2^5}^r\cup\Gamma_{2^5}^{5r} & \text{yes} & 6 \\
C_3\otimes R_5^2(r,1,\chi)_1 & \chi\not\in\mathfrak{B};r=1,3 & 12 & \Gamma_{2^2}\cup\Gamma_{2^3}^{3r}\cup\Gamma_{2^3}^{7r}\cup\Gamma_{2^5}^r\cup\Gamma_{2^5}^{5r} \\
\hline
\end{array}
\]

\begin{landscape}

\[
\begin{array}{llllll}
\multicolumn{6}{l}{\textbf{Table 8. }\text{Irreducible representations of }\mathrm{SL}_2(2,\mathbb{Z}/2^\lambda\mathbb{Z})\text{ for }\lambda>5}\\
\hline \text{type of rep.} & & \text{dim} & \mathfrak{t}-\text{spectrum}& \text{m.f.} &|\text{Gal}| \\\hline 
D_\lambda(\chi) &  & 3\cdot2^{\lambda-1} & \Phi_{2^\lambda}
 &\text{no} &4(\lambda-1) \\
N_\lambda(\chi) &  & 2^{\lambda-1} & \Gamma_{2^\lambda}
 &\text{yes} &4 \\
R_\lambda^0(1,3,\chi) & & 3\cdot2^{\lambda-2} & \Gamma_{2^\lambda}\cup\Gamma_{2^{\lambda-2}}&\text{no} & 8 \\
R_\lambda^0(1,7,\chi) & & 3\cdot2^{\lambda-2} & \Gamma_{2^\lambda}\cup\Phi_{2^{\lambda-3}} &\text{no} &4(\lambda-3) \\
R_\lambda^0(r,t,\chi) & r=1,3;t=1,5 & 3\cdot2^{\lambda-3} & \Gamma^r_{2^\lambda}\cup\Gamma^{5r}_{2^\lambda}\cup\left\{\begin{array}{ccc}\Gamma^r_{2^{\lambda-1}}\cup\Gamma^{5r}_{2^{\lambda-1}} &:&t=1 \\ \Gamma^{3r}_{2^{\lambda-1}}\cup\Gamma^{7r}_{2^{\lambda-1}} &:&t=5\end{array}\right. & \text{yes} & 4 \\
R_\lambda^1(r,t,\chi) & r=1,3;t=3,7 & 3\cdot2^{\lambda-3} & \Gamma^r_{2^\lambda}\cup\Gamma^{5r}_{2^\lambda}\cup\Gamma^{rt}_{2^{\lambda-1}}\cup\Gamma^{7rt}_{2^{\lambda-1}}  &\text{yes} &4 \\
R_\lambda^1(r,t,\chi) & r,t\in\{1,5\} & 3\cdot2^{\lambda-3} & \Gamma^r_{2^\lambda}\cup\Gamma^{5r}_{2^\lambda}\cup\Gamma^{rt}_{2^{\lambda-1}}\cup\Gamma^{3rt}_{2^{\lambda-1}}  &\text{yes} &4 \\
R_\lambda^2(r,1,\chi) & r=1,3 & 3\cdot2^{\lambda-3} & \Gamma^r_{2^\lambda}\cup\Gamma^{5r}_{2^\lambda}\cup\Gamma^{rt}_{2^{\lambda-2}}\cup\Gamma^{5rt}_{2^{\lambda-2}}\cup\Gamma^r_{2^{\lambda-3}}\cup\Gamma^{5r}_{2^{\lambda-3}} &\text{no} & 6 \\
R_\lambda^2(r,3,\chi) & r=1,3 & 3\cdot2^{\lambda-3} & \Gamma^r_{2^\lambda}\cup\Gamma^{5r}_{2^\lambda}\cup\Gamma^{rt}_{2^{\lambda-2}}\cup\Gamma^{5rt}_{2^{\lambda-2}}\cup\Gamma_{2^{\lambda-4}} &\text{no} & 8 \\
R_\lambda^2(r,5,\chi) & r=1,3 & 3\cdot2^{\lambda-3} & \Gamma^r_{2^\lambda}\cup\Gamma^{5r}_{2^\lambda}\cup\Gamma^{rt}_{2^{\lambda-2}}\cup\Gamma^{5rt}_{2^{\lambda-2}}\cup\Gamma^{3r}_{2^{\lambda-3}}\cup\Gamma^{7r}_{2^{\lambda-3}} &\text{no} & 6 \\
R_\lambda^2(r,7,\chi) & r=1,3 & 3\cdot2^{\lambda-3} & \Gamma^r_{2^\lambda}\cup\Gamma^{5r}_{2^\lambda}\cup\Gamma^{rt}_{2^{\lambda-2}}\cup\Gamma^{5rt}_{2^{\lambda-2}}\cup\Phi_{2^{\lambda-5}} & \text{no} & \left\{\begin{array}{ccc} 6 & : & \lambda=6 \\ 4(\lambda-6) & : & \text{else}\end{array}\right. \\
R_\lambda^\sigma(r,t,\chi) & \sigma=3,\ldots,\lambda-3 & 3\cdot2^{\lambda-4} & \{\zeta_{p^\lambda}^{r(x^2+2^\sigma y^2)}:x,y\in\mathbb{Z},x\text{ or }y\text{ is odd}\} & \text{no} & \geq\sigma+1 \\
& r,t\in\{1,3,5,7\} & &  & & \\
R_\lambda^{\lambda-2}(r,t,\chi) & r=1,3,5,7;t=1,3 & 3\cdot2^{\lambda-4} &\Gamma^r_{2^\lambda}\cup\Gamma^r_{2^{\lambda-2}}\cup\Gamma^r_{2^{\lambda-4}}\cup\Phi^r_{2^{\lambda-6}} & \text{no} & \lambda-2 \\\hline
&\text{For }\lambda\geq7\ldots&&&&\\
R_\lambda^{\lambda-3}(r,t,\chi_{\pm1})_1 & r=1,3,5,7;t=1,3 & 3\cdot2^{\lambda-4} & \Gamma^r_{2^\lambda}\cup\Gamma^r_{2^{\lambda-2}}\cup\Gamma^r_{2^{\lambda-4}}\cup\Phi^r_{2^{\lambda-6}} & \text{no} & \lambda-2 \\\hline
&\text{For }\lambda=6\ldots&&&&\\
R_6^4(r,1,\chi_1)_1 & r=1,3,5,7;t=1,3 & 12& \Gamma_1\cup\Gamma_{2^2}^r\cup\Gamma_{2^4}^{5r}\cup\Gamma_{2^6}^r & \text{yes} & 4 \\
C_2\otimes R_6^4(r,t,\chi_1)_1 & r=1,3,5,7;t=1,3 & 12 & \Gamma_2\cup\Gamma_{2^2}^{3r}\cup\Gamma_{2^4}^{5r}\cup\Gamma_{2^6}^r & \text{yes} & 4 \\
\hline
\end{array}
\]
\end{landscape}

\end{appendices}


\bibliographystyle{plain}
\bibliography{bib}

\end{document}